\documentclass[12pt]{amsart}%

\usepackage{graphicx}
\usepackage[T1]{fontenc}
\usepackage{amsmath,amssymb}
\usepackage{tikz}
\usepackage{pgf}
\usepackage{hyperref}
\usepackage{ulem}
\usepackage{geometry}
\usepackage{enumerate,enumitem}
\usepackage[pagewise]{lineno}
\usepackage{amsfonts}
\usepackage{mathptmx}      

\newtheorem{theorem}{Theorem}[section]

\newtheorem{lemma}[theorem]{Lemma}
\newtheorem{corollary}[theorem]{Corollary}
\newtheorem{remark}[theorem]{Remark}
\newtheorem{definition}[theorem]{Definition}
\newtheorem{proposition}[theorem]{Proposition}

\begin{document}

\title{The existence of isolating blocks for multivalued semiflows}
\thanks{The first author has been supported by CAPES, FAPESP Grant\#2018/00065-9 and FAPESP Grant\#2020/00104-4.\\The second author has been partially supported by the Spanish Ministry of Science, Innovation and Universities, project PGC2018-096540-B-I00, by the Spanish Ministry of Science and Innovation, project PID2019-108654GB-I00, by the Junta de Andalucía and FEDER, project P18-FR-4509, and by the Generalitat Valenciana, project PROMETEO/2021/063.}



\author{Estefani M. Moreira}
\address{Departamento de Matem\'{a}tica, Instituto de Ci\^{e}ncias
              Ma\-te\-m\'{a}\-ti\-cas e de Computa\c{c}\~{a}o, Universidade de S\~{a}o
              Paulo, Campus de S\~{a}o Carlos, Caixa Postal 668, S\~{a}o Carlos, 13560-970, Brazil}
              \email{estefani@usp.br}           

\author{Jos\'{e} Valero}
\address{Centro de Investigaci\'on Operativa, Universidad Miguel Hern\'andez de Elche,
              Avenida Universidad s/n, 03202 Elche, Spain}
              \email{jvalero@umh.es}



\begin{abstract}
In this article, we show the existence of an isolating block, a special
neighborhood of an isolated invariant set, for multivalued semiflows acting on
metric spaces (not locally compact). Isolating blocks play an important role
in Conley's index theory for single-valued semiflows and are used to define
the concepts of homology index. Although Conley's index was generalized in the
context of multivalued (semi)flows, the approaches skip the traditional
construction made by Conley, and later, Rybakowski.

Our aim is to present a theory of isolating blocks for multivalued semiflows
in which we understand such a neighborhood of a weakly isolated invariant set
in the same way as we understand it for invariant sets in the single-valued scenario.

After that, we will apply this abstract result to a differential inclusion in
order to show that we can construct isolating blocks for each equilibrium of
the problem.

\end{abstract}

\maketitle

\section{Introduction}
\label{intro}

The last century was core of the development of the theory of nonlinear
dynamical systems. It was necessary to understand the asymptotic behavior of
solutions for problems coming from different areas such as Physics, Biology,
Chemistry and others. However, a model can be just a shadow of the real system
and not necessarily has to give a global and good representation for it. In
this context, the knowledge of some structures in the model that are stable
under perturbations, that is, structures that are also present in problems
which are sufficiently close, plays an essential role in understanding the dynamics.

In the study of nonlinear systems, we may find bounded or unbounded regions
for which the dynamics is self-contained, that is, it is not affected by the
dynamics outside. Such sets are called invariant regions. Isolated invariant
sets are invariant regions that are \textquotedblleft
isolated\textquotedblright\ in some sense. They are subjects of interest
because they are preserved under perturbations. Some special isolated
invariant sets are given by hyperbolic equilibria and hyperbolic periodic
solutions. Proving hyperbolicity for an equilibrium or a periodic solution can
be a very challenging subject, because it requires a good understanding of the
geometrical properties of trajectories in a neighborhood of such elements.

Thus, in the last century, the authors started to create a new approach of
seeking properties that are robust under perturbations. In this sense, Conley
offered us the concept of Conley's index, a way of studying a neighborhood of
isolated invariant sets from a topological point of view. This concept
generalizes the concept of Morse index, which can be calculated only when we
have a good understanding of the geometric local asymptotic behavior.

The concept of Conley's index appeared in \cite{Conley}, where it was defined
for compact isolated invariant sets under the action of semiflows defined on
locally compact metric spaces. Later, Rybakowski \cite{Rybakowski} generalized
the concept for semiflows defined on metric spaces which are not necessarily
compact. The importance of this topology definition can be measured by the
large amount of studies that followed the above references. Just to cite a few
of them, this concept was used in applications (see for instance,
\cite{Rybakowski,Mischaikow95}), it was also defined for flows on Hilbert
spaces (see \cite{GIP_CI_Hilb,BzGaRs_CI_Hilb,IzMaSta_CI_Hilb}), for
non-autonomous semiflows on Banach spaces \cite{Janig_NonAutCI} and also for
multivalued semiflows (see \cite{DzGg_CI_mult_Hilb,Mrozek90_CI}).

Multivalued problems appeared in the literature in the last century. They
appear in the scenario of equations for which we do not have uniqueness of
solutions of the Cauchy problem. The theory of attractors for such systems is
a very challenging subject and has been developed and applied to multiple
equations over the last years (see e.g. \cite{ARV06,Ball,DKP,MelnikValero98,ZKKVZ} and the references therein among many others).

In the theory of Conley's index developed in \cite{Rybakowski}\ the concept of
isolating block plays a fundamental role. This is a neighborhood of an
isolated invariant set of special kind, in which the boundaries are completely
oriented in some sense, characterizing in this way the stable and unstable
subsets. As far as we know, such construction has not be given yet in the
multivalued setting. In this paper we prove the existence of isolating blocks
for multivalued semiflows defined on metric spaces under rather general
assumptions. This is not a mere generalization, as there are many subtle
details that are quite different in the multivalued situation. After that, we
apply this result to a differential inclusion generated by reaction-diffusion
problems with discontinuous nonlinearities.

\section{Basic definitions}

Let $(X,d)$ be a metric space and denote $P(X)=\{B\subset X:B\neq\emptyset\}$,
while $C(\mathbb{R}^{+},X)$ is the set of all continuous functions from
$\mathbb{R}^{+}$ into $X$.

Consider a multivalued map $G:\mathbb{R}^{+}\times X \to P(X)$, that is, a
function that associates each $(t,x)\in\mathbb{R}^{+}\times X$ to the nonempty
subset $G(t,x)\subset X$.

\begin{definition}
	We say that $G$ is a \emph{multivalued semiflow} if:
	
	\begin{itemize}
		\item[i)] $G(0,x)=x$ for all $x \in X$;
		
		\item[ii)] $G(t+s,x)\subset G(t, G(s,x))$ for all $x \in X$ and $t,s \geq0$.
	\end{itemize}
\end{definition}

The multivalued semiflow is strict if, moreover, $G(t+s,x)=G(t, G(s,x))$ for
all $x \in X$ and $t,s \geq0$.

A function $\phi: \mathbb{R}\to X$ is called a complete trajectory of $G$
through $x \in X$, if $\phi(0)=x$ and $\phi(t+s) \in G(t, \phi(s))$, for all
$t,s \in\mathbb{R}$ with $t\geq0$.

We define $\mathcal{R}\subset C(\mathbb{R}^{+},X)$ to be a set of functions
that satisfy the following properties:

\begin{itemize}
	\item[(K1)] For any $x \in X$, we find $\phi\in\mathcal{R}$ such that
	$\phi(0)=x$;
	
	\item[(K2)] \emph{Translation property}: If $\phi\in\mathcal{R}$, then
	$\phi_{\tau}(\cdot)=\phi(\tau+ \cdot) \in\mathcal{R}$, for all $\tau
	\in\mathbb{R}^{+}$.
	
	\item[(K3)] \emph{Concatenation property}: Given any $\phi_{1}, \phi_{2}
	\in\mathcal{R}$ with $\phi_{1}(s)=\phi_{2}(0)$ for some $s\geq0$, the function
	$\phi\in C(\mathbb{R}^{+}, X)$ given by
	\[
	\phi(t)=%
	\begin{cases}
		\phi_{1}(t), \mbox{ if } t\in[0,s],\\
		\phi_{2}(t-s), \mbox{ if } t\in(s,+\infty),
	\end{cases}
	\]
	also belongs to $\mathcal{R}$.
	
	\item[(K4)] Let $\{\phi_{n}\}_{n \in\mathbb{N}}\subset\mathcal{R}$ be a
	sequence with $\phi_{n}(0)\rightarrow x \in X$. Then, we find $\phi
	\in\mathcal{R}$, $\phi(0)=x$ and such that $\phi_{n} \rightarrow\phi$
	uniformly on compacts of $\mathbb{R}^{+}$.
\end{itemize}

The functions from $\mathcal{R}$ generate the strict multivalued semiflow $G:
\mathbb{R}^{+}\times X\to P(X)$ given by
\[
G(t,x)=\{y \in X: y =\phi(t), \phi\in\mathcal{R}, \phi(0)=x\}.
\]
The functions $\phi\in\mathcal{R}$ are called solutions.

\begin{definition}
	A point $x \in X$ is a fixed point of $\mathcal{R}$, if $\phi\in\mathcal{R} $,
	where $\phi(t)=x$, for all $t\geq0$.
	
	A function $\phi:\mathbb{R}\to X$ is a complete trajectory of $\mathcal{R}$
	if, for any $\tau\in\mathbb{R}$, $\phi(\tau+ \cdot)\big|_{[0,+\infty)}%
	\in\mathcal{R}$.
\end{definition}

\begin{remark}
	Any complete trajectory of $\mathcal{R}$ is a complete trajectory of $G$. The
	converse is true when the trajectory is continuous, see \cite{KKV14}.
\end{remark}

When we are in the multivalued case, there are many ways of defining invariance.

\begin{definition}
	Consider a set $A\subset X$. We say that:
	
	\begin{enumerate}
		\item $A$ is \emph{invariant} if $G(t,A)=A$, for all $t\geq0$;
		
		\item $A$ is \emph{negatively (positively) invariant} if $G(t, A)\subset
		A\ (A\subset G(t,A))$, for all $t\geq0$.
		
		\item $A$ is \emph{weakly invariant} if, for all $x\in A$, we find a complete
		trajectory $\phi$ of $\mathcal{R}$ such that $\phi(0)=x$ and $\phi(t)\in A$,
		for all $t\in\mathbb{R}$.
		
		\item $A$ is \emph{weakly positively invariant} if for every $x\in A$ and
		$t\geq0$ it holds that $G(t,x)\cap A\neq\emptyset.$
		
		\item $A$ is \emph{weakly negatively invariant} if, for all $x\in A$, we find
		a complete trajectory $\phi$ of $\mathcal{R}$ such that $\phi(0)=x$ and
		$\phi(t)\in A$, for all $t\in\mathbb{R}^{-}$.
	\end{enumerate}
\end{definition}

The proof of the following propositions can be found in \cite{HenVal}.

\begin{proposition}
	\label{prop:pos.inv} Suppose that conditions {(K1)}-{(K4)} are verified. For a
	closed subset $A\subset X$, the following statements are equivalent:
	
	\begin{itemize}
		\item[i)] $A$ is weakly positively invariant;
		
		\item[ii)] For each $x \in A$, there is $\phi\in\mathcal{R}$ with $\phi(0)=x$
		and $\phi([0,+\infty))\subset A$.
	\end{itemize}
\end{proposition}

\begin{proposition}
	\label{prop:neg.inv} Suppose that conditions {(K1)}-{(K4)} are verified. Let
	$A\subset X$ be a compact set which is negatively invariant. Then, for each
	$x\in A$, there is a complete trajectory $\phi$ of $\mathcal{R}$ with
	$\phi(0)=x$ and $\phi((-\infty,0])\subset A$.
\end{proposition}

Consider $B\subset X$ and $\phi\in\mathcal{R}$. We define the $\omega$-limit
set of $B$ as
\[
\begin{aligned} \omega(B)= \{y \in X: \exists \{t_n\}_{n \in \mathbb{N}}\subset \mathbb{R}^+, \ t_n\rightarrow +\infty \mbox{ and } \{y_n\}_{n \in \mathbb{N}}\subset X,\\ y_n \in G(t_n,B) \mbox{ and } y_n \rightarrow y\} \end{aligned}
\]
and the $\omega$-limit set of $\phi$ as
\[
\omega(\phi)=\{y\in X:\exists\{t_{n}\}_{n\in\mathbb{N}}\subset\mathbb{R}%
^{+},\ t_{n}\rightarrow+\infty\mbox{ and }\phi(t_{n})\rightarrow y\}.
\]

\begin{definition}
	A closed subset $A\subset X$ is called an {isolated weakly invariant} set if
	$A$ is a weakly invariant set and we find an open neighborhood $U\subset X$ of
	$A$, such that $A$ is the maximal weakly invariant set in $U$.
\end{definition}

Assume {(K1)-(K4)} and let $K$ be a closed, isolated and weakly invariant set.
Let $\mathcal{O}(K)$ be an open neighborhood of $K$. For any $\phi
\in\mathcal{R},\phi(0)\in\mathcal{O}(K)$, denote
\[
t_{\phi}=\sup\{t:\phi([0,t])\subset\mathcal{O}(K)\}.
\]

For any $\phi\in\mathcal{R}$, $\phi(0)=x \in\mathcal{O}(K)$ and sequence
$x_{n} \rightarrow x$, $\phi_{n} \in\mathcal{R}$, $\phi_{n}(0)=x_{n}
\in\mathcal{O}(K)$, the convergence $\phi_{n} \rightarrow\phi$ means that
\[
\phi_{n}(s)\rightarrow\phi(s) \mbox{ uniformly on } \lbrack0,t] \mbox{ for }
t<t_{\phi}.
\]

For what we are going to do, we need to ask an additional assumption for
$\mathcal{R}$:

\begin{itemize}
	\item[(K5)] There exists an open neighborhood $\mathcal{O}(K)$ such that, for
	any $x\in\mathcal{O}(K)$ and $\phi\in\mathcal{R}$, with $\phi(0)=x$ and any
	sequence $x_{n}\rightarrow x$ there exists a subsequence $\phi_{n_{k}}%
	\in\mathcal{R}$, $\phi_{n_{k}}(0)=x_{n_{k}}$ such that $\phi_{n_{k}%
	}\rightarrow\phi$ uniformly on compacts sets of $[0,t_{\phi})$.
\end{itemize}

Given a set $V\subset X$, define the sets $\partial V$, $clV$ and $intV$ as,
respectively, the boundary of $V$, the closure of $V$ and the interior of $V$.
To be more precise:
\[
\begin{aligned} &int V = \{ x \in V: \mbox{ there is an open subset }U\subset X \mbox{ with } x \in U\subset V\},\\ &cl V=\{y \in X: \mbox{ for all open subset } U\subset X, \mbox{ with } y \in U, U\cap V\neq \emptyset\},\\ &\partial V = clV \cap cl(X\setminus V). \end{aligned}
\]

\begin{definition}
	Given a closed isolated invariant set $A\subset X$, we say that a closed set
	$N\subset X$ is a related isolating neighborhood if $A\subset int(N)$ (the
	interior of $N$) and $A$ is the maximal isolated weakly invariant set in $N$.
\end{definition}

\begin{definition}
	\label{def:point_eiboff} Let $B\subset X$ be a closed set and $x\in\partial B$
	be a boundary point. We have the following definitions.
	
	\begin{enumerate}
		\item $x$ is an \emph{egress point} if for every $\sigma:[-\delta
			_{1},+\infty)\rightarrow X$, $\sigma_{\delta_{1}}:=\sigma(-\delta_{1}%
			+$\textperiodcentered$)\in\mathcal{R}$, $x=\sigma(0)$, with $\delta_{1}\geq0$,
			the following hold: \subitem There is $\varepsilon_{2}>0$ such that
		$\sigma(t)\notin B$, for $t\in(0,\varepsilon_{2}]$; \subitem If $\delta
		_{1}>0$ then, for some $\varepsilon_{1}\in(0,\delta_{1})$, $\sigma(t)\in B$
		for $t\in\lbrack-\varepsilon_{1},0)$.
		
		\noindent The set of egress points of $B$ is denoted by $B^{e}$.
		
		\item $x$ is an \emph{ingress point} if for every  $\sigma:[-\delta
			_{1},+\infty)\rightarrow X$, $\sigma_{\delta_{1}}\in\mathcal{R}$,
			$x=\sigma(0)$, with $\delta_{1}\geq0$, the following properties hold: \subitem
			There is $\varepsilon_{2}>0$ such that $\sigma(t)\in int(B)$, for
			$t\in(0,\varepsilon_{2}]$; \subitem If $\delta_{1}>0$ then for some
		$\varepsilon_{1}\in(0,\delta_{1})$, $\sigma(t)\notin B$, for $t\in
		\lbrack-\varepsilon_{1},0)$.
		
		\noindent The set of ingress points of $B$ is denoted by $B^{i}$.
		
		\item $x$ is an \emph{bounce-off point} if for every $\sigma:[-\delta
			_{1},+\infty)\rightarrow X$, $\sigma_{\delta_{1}}\in\mathcal{R}$,
			$x=\sigma(0)$, with $\delta_{1}\geq0$, the following properties hold: \subitem
			There is $\varepsilon_{2}>0$ such that $\sigma(t)\notin B$, for
			$t\in(0,\varepsilon_{2}]$; \subitem If $\delta_{1}>0$ then for some
		$\varepsilon_{1}\in(0,\delta_{1})$, $\sigma(t)\notin B$, for $t\in
		\lbrack-\varepsilon_{1},0)$.
		
		\noindent The set of bounce-off points of $B$ is denoted by $B^{b}$.
	\end{enumerate}
\end{definition}

We usually denote $B^{-}=B^{e}\cup B^{b}$ and we call it the exit set of $B$.

\begin{definition}
	Let $K$ be a closed isolated weakly invariant set. An isolating block of $K$
	is a closed set $B\subset X$ which is an isolating neighborhood of $K$ with
	$\partial B=B^{-}\cup B^{i}$ and such that $B^{-}$ is closed.
\end{definition}

Consider the multivalued semiflows $\{G_{n}(t):t\geq0\}$, $n\in\mathbb{N}$,
and $\{G(t):t\geq0\}$, which are generated by the respective sets
$\mathcal{R}_{n}$, $n\in\mathbb{N}$, and $\mathcal{R}$. Assume that $K$ is a
closed set of $X$, which is also a{n }isolated weakly invariant set for the
semiflow $G$. In this situation, we may impose the additional assumptions:

\begin{itemize}
	\item[(KK4)] Let $\{\phi_{n}\}_{n\in\mathbb{N}}$ be a sequence with $\phi
	_{n}\in\mathcal{R}_{n}$, $n\in\mathbb{N}$, and $\phi_{n}(0)\rightarrow x$, for
	some $x\in X$. Then, we find $\phi\in\mathcal{R}$, $\phi(0)=x$ such that
	$\phi_{n}\rightarrow\phi$ uniformly on compacts of $\mathbb{R}^{+}$.
	
	\item[(KK5)] There exists an open neighborhood $\mathcal{O}(K)$ such that, for
	any $x\in\mathcal{O}(K)$ and $\phi\in\mathcal{R}$, with $\phi(0)=x$ and any
	sequence $x_{n}\rightarrow x$ there exists a subsequence $\phi_{n_{k}}%
	\in\mathcal{R}_{n_{k}}$, $\phi_{n_{k}}(0)=x_{n_{k}}$ such that $\phi_{n_{k}%
	}\rightarrow\phi$ uniformly on compacts sets of $[0,t_{\phi})$.
\end{itemize}

Observe that if we consider $\mathcal{R}_{n}\equiv\mathcal{R}$, $n\in
\mathbb{N}$, satisfying $(KK4)$ (resp. $(KK5)$), then $\mathcal{R}$ satisfies
$(K4)$ (resp. $(K5)$).

\begin{definition}
	A closed set $N\subset X$ is called $\{G_{n}\}-$admissible if, given sequences
	$\{x_{n}\}_{n\in\mathbb{N}}$, $\{t_{n}\}_{n\in\mathbb{N}}\in\mathbb{R}^{+}$,
	where $t_{n}\rightarrow+\infty$, and $\{\phi_{n}\}_{n\in\mathbb{N}}$ for which
	$\phi_{n}\in\mathcal{R}_{n}$, $\phi_{n}(0)=x_{n}$, $\phi_{n}([0,t_{n}])\subset
	N$, $n\in\mathbb{N}$, we have that $\{\phi_{n}(t_{n})\}_{n\in\mathbb{N}}$ has
	a convergent subsequence.
	
	We say that $N\subset X$ is $G$-admissible if it is $\{G_{n}\}-$admissible,
	where $G_{n}=G$, for all $n\in\mathbb{N}$.
\end{definition}

\begin{remark}
	We can change the hypothesis of admissibility of the semiflows by the
	collectively asymptotic compactness, if we restrict the analysis to bounded
	isolating neighborhoods $N$. In fact, collectively asymptotic compactness is
	stronger than admissibility.
	
	We recall that $\{G_{n}(t):t\geq0\}_{n\in\mathbb{N}}$ is collectively
	asymptotic compact if, given any sequences $\{t_{n}\}_{n\in\mathbb{N}}$ and
	$\phi_{n}\in\mathcal{R}_{n}$, such that $\{\phi_{n}(0)\}_{n\in\mathbb{N}}$ is
	bounded and $t_{n}\rightarrow+\infty$, we have that $\{\phi_{n}(t_{n}%
	)\}_{n\in\mathbb{N}}$ has a convergent subsequence.
\end{remark}

\section{Existence of the isolating block in the multivalued case}

\label{Block}

In the multivalued case, there is the possibility of having more than one
solution passing through a point and, therefore, we need to adapt the results
related to the existence of the isolating block, which are not directly applicable.

In this section, we will follow the construction made by Rybakowski in
\cite{Rybakowski}, with the necessary adaptations, in order to construct the
isolating block. This means that, although we do not have uniqueness of
solutions, we are still able to construct a closed neighborhood for which its
boundary describes the entry and exit directions.

{Consider $K$ to be a closed isolated }weakly invariant set{ and $N$ to be a
	closed isolating neighborhood of $K$. Denote $U=int(N)$ and define the sets
	$\mathcal{U}=\{\phi\in\mathcal{R}:\phi(0)\in U\}$ and $\mathcal{N}=\{\phi
	\in\mathcal{R}:\phi(0)\in N\}.$ Observe that $\mathcal{U}\subset\mathcal{N}$.}

For the multivalued semiflow $G$ generated by $\mathcal{R}$, we denote by
{$A_{G}^{+}(N)$ the set of points $y\in N$ for which we find a $\phi
	\in\mathcal{R}$ such that $\phi([0,+\infty))\subset N$ and $\phi(0)=y$.} Also,
denote by $A_{G}^{-}(N)$ the set of points $y\in N$ for which we find a
complete trajectory $\phi$ of $\mathcal{R}$ such that $\phi((-\infty
,0])\subset N$ and $\phi(0)=y$. Obviously, $K\subset${$A_{G}^{+}(N)\cap
	A_{G}^{-}(N)$. Using }$\left(  K3\right)  $ we obtain that the converse is
also true, so $K=${$A_{G}^{+}(N)\cap A_{G}^{-}(N).$}

The following result is analogous to Theorem 4.5 in \cite{Rybakowski}.

\begin{proposition}
	\label{RybTheo4.5} Let $N \subset X$ be closed and $G,G_{n}$ be multivalued
	semiflows, $n \in\mathbb{N}$. Denote by $\mathcal{R}_{n}$ the set of functions
	related to $G_{n}$, $n \in\mathbb{N}$, and by $\mathcal{R}$ the set of
	functions related to $G$, that satisfy the properties {(K1)-(K4)} and
	collectively satisfy (KK4).
	
	Consider $x \in X$ and $\{x_{n}\}_{n \in\mathbb{N}}\in X$ with $x_{n}
	\rightarrow x$ as $n \rightarrow+\infty$. Suppose that, for each $n
	\in\mathbb{N}$, we find $\phi_{n} \in\mathcal{R}_{n}$ and $t_{n} \in
	\mathbb{R}^{+}$ such that $\phi_{n}(0)=x_{n}$ and $\phi_{n}([0,t_{n}])\subset
	N$. It follows:
	
	\begin{itemize}
		\item[(a1)] If $t_{n}\rightarrow+\infty$, then we find a
			subsequence of $\{\phi_{n}\}$ and $\phi\in\mathcal{R}$ such that $\phi(0)=x,$
			$\phi(t)\in N$, for all $t\in\mathbb{R}^{+}$, and $\phi_{n}\rightarrow\phi$
			uniformly on compact sets of $\mathbb{R}^{+}$.
		
		\item[(a2)] If $t_{n}\rightarrow t_{0}$, for some $t_{0}\in\mathbb{R}^{+}$,
		then we find a subsequence of $\{\phi_{n}\}$ and $\phi
			\in\mathcal{R}$ such that $\phi(0)=x,$ $\phi([0,t_{0}])\subset N$ and
			$\phi_{n}\rightarrow\phi$ uniformly on compact sets of $\mathbb{R}^{+}$.
	\end{itemize}
	
	Assume that $N$ is $\{G_{n_{m}}\}-$admissible for every subsequence of
	$\{G_{n}\}_{n\in\mathbb{N}}$. Then,
	
	\begin{enumerate}
		\item[(b1)] Every limit point of $\{\phi^{(n)}(t_{n})\}_{n \in\mathbb{N}}$
		belongs to $A^{-}_{G}(N)$.
		
		\item[(b2)] Denote by $K_{n}$ the largest weakly invariant set
			for $G_{n}$ in $N$, $n\in\mathbb{N}$. If $W\subset N$ with $K\subset intW$,
		then $K_{n}\subset intW$ for $n$ sufficiently large. 
		
		\item[(b3)] If $N$ is $G$-admissible, the sets $K$ and $A_{G}^{-}(N)$ are compact.
	\end{enumerate}
\end{proposition}

\begin{proof}
	We prove each statement separately.
	
	\begin{itemize}
		\item[(a1)] By (KK4), we find a $\phi\in\mathcal{R}$ such that $\phi
		_{n}\rightarrow\phi$ uniformly on compact sets of $\mathbb{R}^{+}$. By
		contradiction, suppose that for some $T>0$, $\phi(T)\notin N$. Then we would
		find a $\varepsilon>0$ with $\mathcal{O}_{\varepsilon}(\phi(T))\subset
		X\setminus N$, and $n_{0}\in\mathbb{N}$ sufficiently large such that $\phi
		_{n}(T)\in\mathcal{O}_{\varepsilon}(\phi(T))$, for all $n\geq n_{0}$.
		Consequently, $t_{n}<T$, for all $n\geq n_{0}$. This is a contradiction, since
		$t_{n}\rightarrow+\infty$ as $n$ goes to $+\infty$. Therefore, for all
		$t\in\mathbb{R}^{+}$, we have $\phi(t)\in N$.
		
		\item[(a2)] By (KK4), there is a $\phi\in\mathcal{R}$ for which $\phi
		_{n}\rightarrow\phi$ uniformly on compact sets of $\mathbb{R}^{+}$. For any
		$\varepsilon\in(0,t_{0})$, repeating the same argument above taking
		$T=t_{0}-\varepsilon$, it follows that $\phi([0,t_{0}-\varepsilon])\subset N$.
		Since $N$ is closed, we conclude that $\phi([0,t_{0}])\subset N$.
		
		\item[(b1)] By the admissibility of $N$, we find a subsequence $\{\phi
		_{n}^{(0)}(t_{n}^{(0)})\}_{n\in\mathbb{N}}\subset\{\phi_{n}(t_{n}%
		)\}_{n\in\mathbb{N}}$ such that $\phi_{n}^{(0)}(t_{n}^{(0)})\rightarrow x$ as
		$n\rightarrow+\infty$. We can apply the same arguments for the sequence
		$\{\phi_{n}^{(0)}(t_{n}^{(0)}-1)\}_{n\geq n_{0}}$, where $n_{0}\in\mathbb{N}$
		is such that $t_{n}^{(0)}\geq1$, for $n\geq n_{0}$. Then, we extract a
		subsequence $\{\phi_{n}^{(1)}(t_{n}^{(1)}-1)\}_{n\geq n_{0}}\subset\{\phi
		_{n}^{(0)}(t_{n}^{(0)}-1)\}_{n\geq n_{0}}$ such that $\phi_{n}^{(1)}%
		(t_{n}^{(1)}-1)\rightarrow y_{1}\in N$. We construct a function $\psi
		^{(1)}:[-1,+\infty)\rightarrow X$ with $\psi^{(1)}(-1)=y_{1}$ and $\psi
		^{(1)}(0)=x$, $\psi^{(1)}(\cdot-1)\in\mathcal{R}$ and $\psi^{(1)}%
		([-1,0])\subset N$ using the same construction applied in the proof of Lemma 5
		from \cite{HenVal}.
		
		Applying this argument recursively, for $k\in\mathbb{N}$, we find a
		subsequence
		\[
		\left\{  \phi_{n}^{(k+1)}(t_{n}^{(k+1)}-k-1)\right\}  _{n\in\mathbb{N}}%
		\subset\left\{  \phi_{n}^{(k)}(t_{n}^{(k)}-k-1)\right\}  _{n\in\mathbb{N}}%
		\]
		such that $\phi_{n}^{(k+1)}(t_{n}^{(k+1)}-k-1)\rightarrow y_{k+1}$ and we use
		the construction applied in Lemma 5, \cite{HenVal}, to construct a connection
		between $y_{k+1}$ and $y_{k}$ and concatenation to construct $\psi
		^{(k+1)}:[-k-1,+\infty)\rightarrow X$, with $\psi^{k+1}(\cdot-k-1)\in
		\mathcal{R}$ such that $\psi^{(k+1)}([-k-1,0])\subset N$, $\psi^{(k+1)}%
		(-k-1)=y_{k+1}$ and $\psi^{(k+1)}(0)=x$. Moreover, $\psi^{k+1}(t)=\psi^{k}%
		(t)$, for $t\geq-k$.
		
		Then we can define $\psi:\mathbb{R}\rightarrow X$ with $\psi(t)=\psi^{(k)}(t)$
		if $t\geq-k$. By construction we have that $\psi$ is well-defined, $\psi(0)=x$
		and $\psi(t)\in N$ for all $t\in(-\infty,0]$. Also, $\psi$ is a complete
		trajectory of $\mathcal{R}$. Therefore, we conclude that $x\in A^{-}(N)$, as desired.
		
		\item[(b2)] Suppose that we cannot find $n_{0}\in\mathbb{N}$ in those
		conditions. Then, we can assume, w.l.g., that we find $x_{n}\in K_{n}\cap
		N\setminus intW$, for all $n\in\mathbb{N}$. By definition of $K_{n}$
		there exists a complete trajectory $\phi_{n}:\mathbb{R}\rightarrow
			X$ of $\mathcal{R}_{n}$ with $\phi_{n}(0)=x_{n}$ and $\phi_{n}(\mathbb{R}%
			)\subset N$. Consider the sequence $\{t_{n}\}_{n\in\mathbb{N}}\subset
			\mathbb{R}^{+}$ such that $t_{n}\rightarrow+\infty$. Define $\psi
		_{n}:\mathbb{R}^{+}\rightarrow X$ given by $\psi_{n}(t)=\phi_{n}(t-t_{n})$,
		for all $t\geq0$. {B}y construction, $\psi_{n}([0,t_{n}])\subset N$ and, by
		the admissibility, we find that $\{\psi_{n}(t_{n})\}_{n\in\mathbb{N}}%
		=\{x_{n}\}_{n\in\mathbb{N}}$ has a convergent subsequence to a point $x_{0}\in
		A_{G}^{-}(N)$, by (b1).
		
		By (KK4), since $x_{n}=\phi_{n}(0)\rightarrow x_{0}$, we may assume that we
		find a $\phi\in\mathcal{R}$ such that $\phi(0)=x_{0}$ and $\phi_{n}%
		(t)\rightarrow\phi(t)$ for all $t\geq0$. Now, since $\phi_{n}$ converges to
		$\phi$ and $N$ is closed we easily obtain that $\phi(t)\in N$ for all $t\geq0$.
		
		Hence, $x_{0}\in A_{G}^{-}(N)\cap A_{G}^{+}(N)=K$. But, at the same time,
		$x_{n}\rightarrow x_{0}$, $x_{0}\in cl(N\setminus intW)=N\setminus intW$,
		which is a contradiction, since $K\subset intW$. Therefore, $K_{n}\subset
		intW$, for $n$ sufficiently large.
		
		\item[(b3)] We want to show now that if $N$ is $G$-admissible, then $A_{G}%
		^{-}(N)$ and $K$ are compact.
		
		Consider $\{x_{n}\}_{n\in\mathbb{N}}\in A_{G}^{-}(N)$. Take a sequence
		$\{t_{n}\}_{n\in\mathbb{N}}\in\mathbb{R}^{+}$ with $t_{n}\rightarrow+\infty$
		as $n$ goes to $+\infty$. By definition of $A_{G}^{-}(N)$, for each
		$n\in\mathbb{N}$, we find a complete trajectory $\psi_{n}$ in $\mathcal{R}$
		with $\psi_{n}((-\infty,0])\subset N$ and $\psi_{n}(0)=x_{n}$. For each
		$n\in\mathbb{N}$, denote $\xi_{n}:[0,+\infty)\rightarrow X$ as $\xi_{n}%
		(\cdot)=\phi_{n}\big|_{[0,+\infty)}(\cdot-t_{n})\in\mathcal{R}$. Then we have
		that $\xi_{n}(t_{n})=x_{n}$ and $\xi_{n}([0,t_{n}])\subset N$, for all
		$n\in\mathbb{N}$. The compactness of $A_{G}^{-}(N)$ will follow
			by the admissibility and applying item (b1) to the sequence $\{\xi
		(t_{n})\}_{n\in\mathbb{N}}$ in the particular case $G_{n}\equiv G$,
		$n\in\mathbb{N}$.
		
		Suppose now that $\{x_{n}\}_{n\in\mathbb{N}}\subset K$. We can assume that
		$x_{n}\rightarrow x\in A_{G}^{-}(N)$, since $K\subset A_{G}^{-}(N)$ and
		$A_{G}^{-}(N)$ is compact. Our goal is to show that we can find $\phi
		\in\mathcal{R}$ such that $\phi(0)=x$ and $\phi(t)\in N$, for all $t\geq0$. We
		just observe that, for each $n\in\mathbb{N}$, we can find 
			$\phi_{n}\in\mathcal{R}$ such that $\phi_{n}(0)=x_{n}$ and $\phi
		_{n}(\mathbb{R}^{+})\subset N$. Then, by the property (K4), we
		find $\phi\in\mathcal{R}$ such that $\phi(0)=x$ and $\phi_{n}(t)\rightarrow
		\phi(t)$, for all $t\geq0$. Since $N$ is closed, it follows that
		$\phi(\mathbb{R}^{+})\subset N$, hence $x\in A_{G}^{+}(N)$. Therefore, $x\in
		K$ and $K$ is compact.
	\end{itemize}
\end{proof}

\bigskip

Before presenting the main theorem of this section, we need to define
auxiliary functions that play an essential role in the definition of the
block. So, we will fix a multivalued semiflow $G$, the sets $\emptyset\neq
U\subset N$ with $U=intN$ and we will denote $A_{G}^{-}(N)$ simply by
$A^{-}(N)$. We recall that {$\mathcal{U}=\{\phi\in\mathcal{R}:\phi(0)\in U\}$
	and $\mathcal{N}=\{\phi\in\mathcal{R}:\phi(0)\in N\}.$}

Basically, we will construct two functions, one \textquotedblleft
identifying\textquotedblright\ the stable part inside the neighborhood $N$ and
the other \textquotedblleft identifying\textquotedblright\ the unstable part
inside the neighborhood $N$. For these functions, we will prove results of
monotonicity and upper semicontinuity that will be important to characterize
the flow behavior close to the invariant set.

Define the functions:

\begin{enumerate}
	\item $s^{+}: \mathcal{N} \to\mathbb{R}\cup\{\infty\}$,
	\[
	s^{+}(\phi)=\sup\{t \in\mathbb{R}^{+}: \phi\mbox{ is defined on } \lbrack0,t]
	\mbox{ and } \phi([0,t])\subset N\},
	\]

	\item $t^{+}: \mathcal{U}\to\mathbb{R}\cup\{\infty\}$,
	\[
	t^{+}(\phi)=\sup\{t \in\mathbb{R}^{+}: \phi\mbox{ is defined on } \lbrack0,t]
	\mbox{ and } \phi([0,t])\subset U\},
	\]

	\item $F: X \to[0,1]$, $F(x)=\min\{1,d(x,A^{-}(N))\}$,
	
	\item $D: X \to[0,1]$, $D(x)=\frac{d(x,K)}{d(x,K)+d(x, X\setminus N)}$,
	
	\item $g^{+}: U \to\mathbb{R}^{+}$,
	\[
	g^{+}(x)=\inf_{{\phi\in\mathcal{U},}{\phi(0)=x}} \inf\left\{  \frac
	{D(\phi(t))}{1+t}: 0\leq t< t^{+}(\phi)\right\}  ,
	\]

	\item $g^{-}: N \to\mathbb{R}^{+}$,
	\[
	g^{-}(x)=\sup_{{\phi\in\mathcal{N},}{\phi(0)=x}} \sup\left\{  \alpha
	(t)F(\phi(t)):
	\begin{aligned} 0\leq t \leq s^+(\phi), \mbox{ if } s^+(\phi)<\infty,\\ 0\leq t <+\infty, \mbox{ otherwise} \end{aligned}\right\}
	,
	\]
	where $\alpha:[0,\infty)\to[1,2)$ is a monotone increasing $C^{\infty}-$diffeomorphism.
\end{enumerate}

Observe that the functions $F$ and $D$ are continuous since they can be
written as compositions of continuous functions. In that situation, we find
the following result.

\begin{proposition}
	\label{prop:monot.gplus}{Assume (K1)-(K4)}. The function $g^{+}$ is increasing
	along solutions as long the orbits are in $U$. To be more precise, given
	$x,y\in U$, for which there are $\varphi\in\mathcal{R}$ and $t>0$, with
	$\varphi(0)=x$, $y=\varphi(t)$ and $\varphi([0,t])\subset U$, it follows that
	$g^{+}(x)\leq g^{+}(y)$.
	
	Moreover, if $g^{+}(x)\neq0$, then $g^{+}(x)< g^{+}(y)$.
\end{proposition}

\begin{proof}
	Consider $x,y\in U$ such that there are $\varphi\in\mathcal{R}$, $t>0$, with
	$\varphi(0)=x$, $\varphi(t)=y$ and $\varphi([0,t])\subset U$. We want to show
	that {$g^{+}(x)\leq g^{+}(y).$}
	
	In order to simplify the exposition, for any $\psi\in\mathcal{U}$ we define
	\[
	f^{+}(\psi)=\inf\left\{  \frac{D(\psi(t))}{1+t}:0\leq t<t^{+}(\psi)\right\}
	\]
	and the sets $\mathcal{U}_{z}=\{\phi\in\mathcal{U}:\phi(0)=z\}$, for $z=x,y$.
	Then, for each $\phi\in\mathcal{U}_{y}$ we can define a $\psi\in
	\mathcal{U}_{x}$ as the concatenation of $\varphi$ and $\phi$, which is
	well-defined by (K3). The set of such functions $\psi$ will be denoted by
	$\mathcal{U}_{xy}$. {We will show that $f^{+}(\psi)\leq f^{+}(\phi)$. In
		fact,
		\[
		\begin{aligned} f^+(\psi)& =\inf\left(\left\{\frac{D(\varphi(s))}{1+s}: s\in[0,t]\right\}\cup \left\{\frac{D(\phi(s-t))}{1+s}: s\in [t, t^+(\psi))\right\}\right)\\ &\leq \inf\left\{\frac{D(\phi(s-t))}{1+s}: s\in [t, t^+(\psi))\right\}=\inf\left\{\frac{D(\phi(u))}{1+u+t}: u\in [0, t^+(\phi))\right\}\leq f^+(\phi), \end{aligned}
		\]
		where we have used that $t^{+}(\psi)=t^{+}(\phi)+t$, by construction. }
	
	Thus
	\[
	g^{+}(x)=\inf_{\phi\in\mathcal{U}_{x}} f^{+}(\phi) \leq\inf_{\psi
		\in\mathcal{U}_{xy}} f^{+}(\psi) \leq\inf_{\phi\in\mathcal{U}_{y}} f^{+}%
	(\phi)=g^{+}(y).
	\]

	Now, assume that $g^{+}(x)\neq0$. Thus, we find $\mu>0$ such
		that $f^{+}(\phi_{t})\geq f^{+}(\phi)\geq\mu$, for $\phi_{t}(\cdot
		)=\phi(t+\cdot)$, $\phi\in\mathcal{U}_{x}$. Since
		$f^{+}(\phi_{t})>0$, there is $\tau=\tau(\phi)\geq t$ for which $f^{+}%
	(\phi_{t})=\frac{D(\phi_{t}(\tau-t))}{1+(\tau-t))}$.
	
	Hence,
	\[
	\begin{aligned} \frac{f^+(\phi_t)-f^+(\phi)}{t} \geq \frac{1}{t}\left(\frac{D(\phi(\tau))}{1+(\tau-t)}-\frac{D(\phi(\tau))}{1+\tau}\right) \geq \frac{D(\phi(\tau))}{(1+\tau)^2} \geq \frac{f^+(\phi)}{1+\tau} \geq \frac{\mu}{1+\tau}. \end{aligned}
	\]

	We claim that there is $\mu_{0}>0$ such that $\frac{\mu}{1+\tau(\phi)}\geq
	\mu_{0}$, for all $\phi\in\mathcal{U}_{xy}$. If this were not true, we would
	find a sequence $\phi^{(n)}\in\mathcal{U}_{xy}$ and $\{\tau_{n}\}_{n\in
		\mathbb{N}}\in\lbrack t,+\infty)$ with $\tau_{n}\rightarrow+\infty$ and such
	that $f^{+}(\phi_{t}^{(n)})=\frac{D(\phi_{t}^{(n)}(\tau_{n}-t))}{1+(\tau
		_{n}-t)}$. But then we would have
	\[
	0<g^{+}(x)\leq g^{+}(y)=\inf_{\phi\in\mathcal{U}_{xy}}f^{+}(\phi_{t})\leq
	\inf_{n\in\mathbb{N}}f^{+}(\phi_{t}^{(n)})=\inf_{n\in\mathbb{N}}\tfrac
	{D(\phi_{t}^{(n)}(\tau_{n}-t))}{1+\tau_{n}-t}\leq\inf_{n\in\mathbb{N}}%
	\tfrac{1}{1+\tau_{n}-t}=0,
	\]
	which is a contradiction. Therefore, we find $\mu_{0}>0$ which {implies that
		$g^{+}(x)+\mu_{0}t\leq g^{+}(y)$} and, consequently, $g^{+}(x)<g^{+}(y)$, as desired.
\end{proof}

\bigskip

The above proposition is the analogous to the item (2) in \cite[Proposition
	5.2]{Rybakowski} in the single-valued case. Further, we will show the
monotonicity of $g^{-}$.

\begin{proposition}
	\label{prop:monot.gminus} {Assume (K1)-(K4). Let the closed set $N$ be }%
	$G$-{admissible.} The function $g^{-}$ is decreasing along solutions as long
	the orbits stay on $N$. To be more precise, if $x,y\in N$ and there are
	$\varphi\in\mathcal{R}$ and $s>0$, with $\varphi(0)=x$, $y=\varphi(s)$ and
	$\varphi([0,s])\subset N$, then $g^{-}(x)\geq g^{-}(y)$.
	
	Moreover, if $g^{-}(x)\neq0$, then $g^{-}(x)> g^{-}(y)$.
\end{proposition}

\begin{proof}
	Let $x,y\in N$, $\varphi\in\mathcal{R}$ and $s>0$ be as in the hypothesis.
	Consider, for any $\phi\in\mathcal{N}$,
	\[
	f^{-}(\phi)=\sup\left\{  \alpha(t)F(\phi
	(t)):\begin{aligned} 0\leq t \leq s^+(\phi), \mbox{ if } s^+(\phi)<\infty,\\ 0\leq t <+\infty, \mbox{ otherwise} \end{aligned}\right\}
	\]
	and the sets $\mathcal{N}_{z}=\{\phi\in\mathcal{N}:\phi(0)=z\}$, for $z=x,y$,
	and $\mathcal{N}_{xy}=\{\phi\in\mathcal{N}:\phi(0)=x\mbox{ and }$ $\phi(s)=y\}$.
	
	For each $\phi\in\mathcal{N}_{y}$, by (K3), we can define a $\psi
	\in\mathcal{N}_{xy}$ as the concatenation of $\varphi$ with $\phi$. In a
	similar way as was done before for the function $f^{+}$, one can prove easily
	that $f^{-}(\psi)\geq f^{-}(\phi)$. Consequently,
	\[
	g^{-}(x)=\sup_{\phi\in\mathcal{N}_{x}}f^{-}(\phi)\geq\sup_{\psi\in
		\mathcal{N}_{xy}}f^{-}(\psi)\geq\sup_{\phi\in\mathcal{N}_{y}}f^{-}(\phi
	)=g^{-}(y),
	\]
	that is, ${g}^{-}(x)\geq g^{-}(y).$
	
	We want to show that, if $g^{-}(x)\neq0$, we have $g^{-}(y) < g^{-}(x)$. By
	definition, $g^{-}(x)=\sup_{\phi\in\mathcal{N}_{x}} f^{-}(\phi)$, which
	implies that, for at least one $\phi\in\mathcal{N}_{x}$, we have $f^{-}%
	(\phi)\neq0$. In particular, $x=\phi(0) \notin A^{-}(N) $ and $F(x)>\mu$, for
	some $\mu>0$.

	In order to obtain the desired result, we will show that there is $\delta>0$
	such that, for all $\phi\in\mathcal{N}_{x}$, we have $f^{-}(\phi)>\delta
	+f^{-}(\phi_{s})$. If this can be proved, we may take supremum in both sides
	and we would find, for $y=\phi(s)$, that
	\[
	g^{-}(x)\geq\sup_{\phi\in\mathcal{N}_{xy}}f^{-}(\phi)\geq\sup_{\phi
		\in\mathcal{N}_{xy}}f^{-}(\phi_{s})+\delta=g^{-}(y)+\delta,
	\]
	which implies $g^{-}(x)>g^{-}(y)$.
	
	If there does not exist such $\delta>0$, then, for each $n \in\mathbb{N}$, we
	would find $\phi^{(n)} \in\mathcal{N}_{x}$ with
	\[
	f^{-}\left(  \phi^{(n)}\right)  \leq f^{-}\left(  \phi_{s}^{(n)}\right)
	+\frac{1}{1+n}.
	\]

	By {(K4)}, since $\{\phi^{(n)}(0)=x\}_{n\in\mathbb{N}}$ is convergent, we may
	assume $\phi^{(n)}(t)\rightarrow\phi(t)$, for all $t\geq0$, for some $\phi
	\in\mathcal{N}_{x}$. So, up to a convergent subsequence, we find
	\[
	0<\mu\leq\beta:=\lim_{n\rightarrow+\infty}f^{-}(\phi^{(n)})\leq\lim
	_{n\rightarrow+\infty}f^{-}(\phi_{s}^{(n)}).
	\]

	Observe that, for each $n\in\mathbb{N}$, we find $t_{n},\ r_{n},\gamma
	_{n},\eta_{n}\in\mathbb{R}^{+}$ such that
$$\begin{aligned}
&f^{-}(\phi^{(n)})=\alpha
	(t_{n})F(\phi^{(n)}(t_{n})) + \gamma_{n}, \\
& f^{-}(\phi_{s}^{(n)})=\alpha
	(r_{n})F(\phi_{s}^{(n)}(r_{n}))+\eta_{n}=\alpha(r_{n})F(\phi^{(n)}%
	(r_{n}+s))+\eta_{n}
\end{aligned}
	$$ and $\gamma_{n},\eta_{n}\rightarrow0.$ Without loss of
	generality, we may assume that $t_{n}\rightarrow t_{0}$ and $r_{n}\rightarrow
	r_{0}$ as $n$ goes to $+\infty$, for some $t_{0},r_{0}\in\mathbb{R}^{+}%
	\cup\{+\infty\}$.
	
	Suppose initially that $t_{n}\rightarrow+\infty$ as $n$ goes to $+\infty$.{
		Since $N$ is admissible} and $\phi^{(n)}([0,t_{n}])\subset N$, by Proposition
	\ref{RybTheo4.5} we find $y\in A^{-}(N)$ such that, up to a subsequence,
	$\phi^{(n)}(t_{n})\rightarrow y$ as $n$ goes to $+\infty$. Consequently,
	$F(\phi^{(n)}(t_{n}))\rightarrow F(y)=0$ and, since $\alpha(t_{n}%
	)\rightarrow\overline{\alpha}\leq2$, we conclude that $\beta=0$, which is a
	contradiction.

	On the other hand, assume that $t_{n}\rightarrow t_{0}<+\infty$. Since
	$\beta\neq0$, we also have $r_{n}\rightarrow r_{0}<+\infty$. It follows that
	$\beta=\alpha(t_{0})F(\phi(t_{0}))\leq\alpha(r_{0})F(\phi(r_{0}+s))$.
	
	But now, for any $n\in\mathbb{N}$, by definition of $t_{n}$, 
	$\alpha(r_{n}+s)F(\phi^{(n)}(r_{n}+s))\leq\alpha(t_{n})F(\phi^{(n)}%
	(t_{n})) + \gamma_{n}$ and, by applying the limit, we obtain%
	\[
	\alpha(r_{0}+s)F(\phi(r_{0}+s))\leq\alpha(r_{0})F(\phi(r_{0}+s))\implies
	\alpha(r_{0}+s)\leq\alpha(r_{0}),
	\]
	since $F(\phi(r_{0}+s))\neq0$. This is a contradiction with $\alpha$ being a
	strictly increasing function.
	
	Therefore, we can find $\delta>0$, as desired, and the assumption follows.
\end{proof}

\bigskip

{The above {result is the analogous of the item (3) in \cite[Proposition
		5.2]{Rybakowski} in the univalued case. Propositions \ref{prop:monot.gplus}
		and \ref{prop:monot.gminus} have shown the monotonicity of $g^{\pm}$ for
		points in the same orbit.} }

The following result describes the role that $g^{+}$ and $g^{-}$ have on
identifying weakly invariant regions.

\begin{proposition}
	{Assume (K1)-(K4)}. Let $K\neq\emptyset$ be a closed isolated weakly invariant
	set. Suppose that $N$ is a closed isolating neighborhood of $K$. The following holds:
	
	\begin{itemize}
		\item[i)] Consider $x\in U$. If $g^{+}(x)=0$, then $x\in A^{+}(N)$.
		
		\item[ii)] Consider $x\in N$. We have that $g^{-}(x)=0$ if, and only if, $x\in
		A^{-}(N)$.
		
		If $x\in N$ and $g^{-}(x)=0$, it follows that $g^{-}(y)=0$ for all values
		$y=\phi(t)$, where $\phi\in\mathcal{R}$ with $\phi(0)=x$ and $t\in
		\lbrack0,s^{+}(\phi))$. If $s^{+}(\phi)<+\infty$, then this is
			true also for $t=s^{+}(\phi)$.
	\end{itemize}
\end{proposition}

\begin{proof}
	Consider item i). If $g^{+}(x)=0$, then, for each $n\in\mathbb{N}$, we find
	$\phi_{n}\in\mathcal{R}$, with $\phi_{n}(0)=x$ and $f^{+}(\phi_{n})<\frac
	{1}{1+n}$. Consequently, for each $n\in\mathbb{N}$, we find $t_{n}\in
	\lbrack0,t^{+}(\phi_{n}))$ with
	\begin{equation}
		\frac{D(\phi_{n}(t_{n}))}{1+t_{n}}<\frac{1}{n+1}. \label{eq:Dphi_n}%
	\end{equation}
	By (K4) we may assume, w. l. g., that as $n$ goes to $+\infty$,
	$\phi_{n}\rightarrow\phi$ uniformly on compact sets, for some $\phi
	\in\mathcal{R}$, with $\phi(0)=x$, and $t_{n}\rightarrow t_{0}$, for some
	$t_{0}\in\mathbb{R}^{+}\cup\{+\infty\}$. In particular, as $N$ is
		closed, it follows that $\phi([0,t_{0}))\subset N$. Now, if $t_{0}=+\infty$,
	then $x\in A^{+}(N)$ follows. On the other hand, if $t_{0}<+\infty$, by
	\eqref{eq:Dphi_n}, we have that $\frac{D(\phi(t_{0}))}{1+t_{0}}=0$. This means
	that $\phi(t_{0})\in K$. In particular, there is $\psi\in\mathcal{R}$ with
	$\psi(0)=\phi(t_{0})$ and $\psi(\mathbb{R}^{+})\subset N$. Consider $\eta
	\in\mathcal{R}$, the concatenation between $\phi$ and $\psi$, which is
	well-defined by the property (K3). Then, $\eta\left(  t\right)  \in N$ for all
	$t\geq0$, so $x\in A^{+}(N)$.
	
	For item ii), suppose first that $g^{-}(x)=0$ for some $x\in N$. Hence, for
	any $\phi\in\mathcal{R}$, with $\phi(0)=x$, we have $f^{-}(\phi)=0$. In
	particular, $F(x)=0$, which implies that $x\in A^{-}(N)$.
	
	Consider now that $x\in A^{-}(N)$ and $\phi\in\mathcal{R}$, with $\phi(0)=x$.
	Fix $t\in\lbrack0,s^{+}(\phi))$, and let $y=\phi(t)$. We want to show that
	$y\in A^{-}(N)$. In fact, since $x\in A^{-}(N)$, there is a
		complete trajectory $\psi$ of $\mathcal{R}$ such that $\psi(0)=x$ and
		$\psi((-\infty,0])\subset N$. Using (K2) and (K3), we have that the
		concatenation of the functions $\psi(t+$\textperiodcentered$)$ and $\phi
		(t+$\textperiodcentered$)$, given by%
		\[
		\varphi(s)=\left\{
		\begin{array}
			[c]{c}%
			\psi(s+t)\text{ if }s\leq-t,\\
			\phi(s+t)\text{ if }s\geq-t,
		\end{array}
		\right.
		\]
		is a complete trajectory of $\mathcal{R}$ and satisfies $\varphi
		((-\infty,0])\subset N$ with $\varphi(0)=y$. Consequently, for any $y\in
		\phi([0,s^{+}(\phi)))$, we have that $y\in A^{-}(N)$ and $f^{-}(\phi)=0$,
		since $t\in\lbrack0,s^{+}(\phi))$ was arbitrary. Therefore, $g^{-}(x)=0$, by
		definition. {If $s^{+}(\phi)<+\infty$, then this is true also for
			$t=s^{+}(\phi)$.} 
	
	The second part follows from the above, since we have shown that the points
	$y\in N$ such as in the hypothesis are also in $A^{-}(N)$.
\end{proof}


 We want to show that the functions $g^{+},g^{-}$ are continuous
	if we restrict their domains.

\begin{lemma}
	\label{lemma:gminus.not.ls} {Assume (K1)-(K4).} Let $K\neq\emptyset$ be a
	closed isolated weakly invariant set. Suppose that $N$ is a closed isolating
	neighborhood of $K$ and that $g^{+}$ is not lower semicontinuous at $x\in
	U=int\ N$. Then there exist sequences $\{x_{n}\}_{n\in\mathbb{N}}\in{U}$,
	$\{\phi_{n}\}_{n\in\mathbb{N}}\in\mathcal{R}$, $\{t_{n}\}_{n\in\mathbb{N}}%
	\in\mathbb{R}^{+}$ and $\phi\in\mathcal{U}_{x}$ satisfying $x_{n}\rightarrow
	x,$ $\phi_{n}\rightarrow\phi$ as $n\rightarrow+\infty$ and, for all
	$n\in\mathbb{N}$, $\phi_{n}(0)=x_{n}$, $t^{+}(\phi)<t_{n}<t^{+}(\phi_{n})$ and
	$\frac{D(\phi_{n}(t_{n}))}{1+t_{n}}<g^{+}(x)$.
\end{lemma}

\begin{proof}
	If $g^{+}$ is not lower semicontinuous at $x\in U$, there are $\mu>0$ and
	$\{x_{n}\}_{n\in\mathbb{N}}\in U$ with $g^{+}(x_{n})<\mu<g^{+}(x)$, for all
	$n\in\mathbb{N}$, and $x_{n}\rightarrow x$ as $n\rightarrow+\infty$. Then, for
	each $n\in\mathbb{N}$, we find $\phi_{n}\in\mathcal{R}$, $\phi_{n}(0)=x_{n}$,
	and $t_{n}\in(0,t^{+}(\phi_{n}))$ such that
	\[
	\mu>\frac{D(\phi_{n}(t_{n}))}{1+t_{n}}.
	\]

	Now, by {(K4)}, we may assume that $\phi_{n}\rightarrow\phi$, for some
	$\phi\in\mathcal{U}_{x}$. Observe that $t^{+}(\phi)<+\infty$, since
	$f^{+}(\phi)\geq g^{+}(x)\neq0$. {Following the proof of Proposition
		\ref{RybTheo4.5}, we have that $\{t^{+}(\phi_{n})\}_{n\geq n_{0}}$ is bounded
		for $n_{0}$ sufficiently large, which assures that the sequence $\{t_{n}%
		\}_{n\geq n_{0}}$ is also bounded.} Without loss of generality, we may assume
	that there is a $t_{0}\in\mathbb{R}^{+}$ for which $t_{n}\rightarrow t_{0}$ as
	$n\rightarrow+\infty$.
	
	Since $\frac{D(\phi_{n}(t_{n}))}{1+t_{n}}\rightarrow\frac{D(\phi(t_{0}%
		))}{1+t_{0}}$, as $n\rightarrow+\infty$, it follows that $\frac{D(\phi
		(t_{0}))}{1+t_{0}}\leq\mu$. Necessarily, we must have $t_{0}>t^{+}(\phi)$ and
	then $t_{n}>t^{+}(\phi)$ for $n$ sufficiently large.
	
	Therefore, we just need to replace $\{x_{n}\}_{n \in\mathbb{N}}$, $\{\phi
	_{n}\}_{n \in\mathbb{N}}$ and $\{t_{n}\}_{n \in\mathbb{N}}$ by proper choices
	of subsequences and the result follows.
\end{proof}

\bigskip

Now, we want to show that, if we restrict the domain of these functions to an
appropriate neighborhood of $K$, $g^{+}$ and $g^{-}$ will be continuous functions.

\begin{proposition}
	\label{prop:properties.g} Assume (K1)-(K4). Let $K\neq
	\emptyset$ be a closed isolated weakly invariant set. Suppose $N$ is a closed
	$G$-admissible isolating neighborhood of $K$. The following statements hold:
	
	\begin{itemize}
		\item[1)] Suppose that we have $\{\phi_{n}\}_{n\in\mathbb{N}},\phi
		\in\mathcal{R}$ with $\phi_{n}(s)\rightarrow\phi(s)$ uniformly for $s$ on
		compacts of $\mathbb{R}^{+}$. If $t^{+}(\phi)>\mu$, for some $\mu>0$, then
		there is $n_{0}\in\mathcal{N}$ such that $t^{+}(\phi_{n})>\mu$, for all
		$n\in\mathbb{N}$, $n\geq n_{0}$. If $s^{+}(\phi)<\tau$, for some constant
		$\tau>0$, then there is $n_{1}\in\mathbb{N}$, such that $s^{+}(\phi_{n})<\tau
		$, for all $n\in\mathbb{N}$, with $n\geq n_{1}$.
		
		\item[2)] Assuming {(K5)}, we can prove that $g^{+}$ is upper-semicontinuous
		in $U\cap\mathcal{O}(K)$. Also, the map $g^{+}$ is continuous in a
		neighborhood of $K$ in $\mathcal{O}(K)\cap U$.
		
		\item[3)] The map $g^{-}$ is upper-semicontinuous in $N$. Assuming {(K5)},
		$g^{-}$ is continuous in any neighborhood $W$ of $K$ in $U\cap\mathcal{O}(K)$
		for which $t^{+}(\phi)=s^{+}(\phi)$, for any $\phi\in\mathcal{U}$ with
		$\phi(0)\in W$.
	\end{itemize}
\end{proposition}

\begin{proof}

	\begin{itemize}
		\item[1)] Since $\mu<t^{+}(\phi)$, we have $\phi([0,\mu])\subset U$. Now, $U$
		is open and, by the uniform convergence of $\{\phi_{n}\}_{n\in\mathbb{N}}$, we
		find $n_{0}\in\mathbb{N}$ such that
		\[
		\phi_{n}([0,\mu])\subset U,\mbox{ for all }n\in\mathbb{N},\ n\geq n_{1}.
		\]
		Hence $\mu<t^{+}(\phi_{n})$, for all $n\geq n_{1}$. The strict inequality
		comes from the fact that either $t^{+}(\phi_{n})=+\infty$ or $t^{+}(\phi
		_{n})<+\infty$ and $\phi_{n}(t^{+}(\phi_{n}))\in\partial U$.
		
		For the second part, assume that there is a $\tau>0$, such that $s^{+}%
		(\phi)<\tau$. Then, by definition of $s^{+}(\phi)$, we can find $T\in
		(s^{+}(\phi),\tau)$ for which $\phi(T)\in X\setminus N$. The set $X\setminus
		N$ is open, since $N$ is closed. Hence, there is an open set $W\subset
		X\setminus N$, such that $\phi(T)\in W$. Since $\phi_{n}(T)\rightarrow\phi(T)$
		as $n\rightarrow+\infty$, we find a $n_{1}\in\mathbb{N}$ for which $\phi
		_{n}(T)\in W$, for all $n\geq n_{1}$. Therefore, $s^{+}(\phi_{n})\leq T<\tau$,
		for all $n\geq n_{1}$.
		
		\item[2)] Consider $x_{0}\in\mathcal{O}(K)\cap U$ and $\mu>g^{+}(x_{0})$. We
		want to show that we can find a neighborhood $W$ of $x_{0}$ in $\mathcal{O}%
		(K)\cap U$ such that, for each $z\in W$, we have $g^{+}(z)<\mu$. If this is
		not true, then we would find a sequence $\{y_{n}\}_{n\in\mathbb{N}}%
		\subset\mathcal{O}(K)\cap U$ such that $y_{n}\rightarrow x_{0}$ and
		$g^{+}(y_{n})\geq\mu$.
		
		Consider $\phi\in\mathbb{R}$, with $\phi(0)=x_{0}$ and $f^{+}(\phi)<\mu$.
		Using {(K5)}, we find a subsequence $y_{n_{k}}$ and $\phi_{k}\in\mathcal{R}$,
		with $\phi_{k}(0)=y_{n_{k}}$ and $\phi_{k}\rightarrow\phi$. {Since $f^{+}%
			(\phi)<\mu$, there is $t_{0}\in\mathbb{R}^{+},$ $t_{0}<t^{+}(\phi)$ such that
			$\frac{D(\phi(t_{0}))}{1+t_{0}}<\mu$. By item 1), there is $k_{0}\in
			\mathbb{N}$ such that $t_{0}<t^{+}(\phi_{k})$ and $\frac{D(\phi_{k}(t_{0}%
				))}{1+t_{0}}<\mu$, for all $k\geq k_{0}$. }Hence, for $k\geq k_{0}$, we have
		\[
		\mu\leq g^{+}(y_{n_{k}})\leq f^{+}(\phi_{k})\leq\frac{D(\phi_{k}(t_{0}%
			))}{1+t_{0}}<\mu
		\]
		which is a contradiction. Therefore, $g^{+}$ is upper semicontinuous at
		$x_{0}\in\mathcal{O}(K)\cap U$.
		
		We also want to show that there is an open neighborhood $W$ of $K$ with
		$W\subset\mathcal{O}(K)\cap U$ for which $g^{+}$ is lower semicontinuous in
		$W$. If this fact is not true, we find a sequence $\{x_{n}\}_{n\in\mathbb{N}%
		}\in\mathcal{O}(K)\cap U$ such that
		\begin{equation}
			d(x_{n},K)\rightarrow0,\ \mbox{ as }n\mbox{ goes to }+\infty, \label{eq:x_nK}%
		\end{equation}
		and $g^{+}$ is not lower semicontinuous at $x_{n}$, for all $n\in\mathbb{N}$.
		By Proposition \ref{RybTheo4.5} the set $K$ is compact, so that we may assume
		that $\{x_{n}\}_{n\in\mathbb{N}}$ converges to some $x_{0}\in K$ as $n$ goes
		to $+\infty$.
		
		Using Lemma \ref{lemma:gminus.not.ls}, for each $n\in\mathbb{N}$, we find:
		
		\begin{itemize}
			\item[i)] A sequence $\{x_{n}^{m}\}_{m\in\mathbb{N}}\subset\mathcal{O}(K)\cap
			U$ with $x_{n}^{m}\rightarrow x_{n}$ as $m$ goes to $+\infty$;
			
			\item[ii)] A sequence of solutions $\phi_{n}^{m}\in\mathcal{R}$, with
			$\phi_{n}^{m}(0)=x_{n}^{m}$, $m\in\mathbb{N}$, and such that $\phi_{n}%
			^{m}\rightarrow\phi_{n}$, for some $\phi_{n}\in\mathcal{R}$ with $\phi
			_{n}(0)=x_{n}$;
			
			\item[iii)] A sequence $\{t_{n}^{m}\}_{m\in\mathbb{N}}\in\mathbb{R}^{+}$ with
			$t^{+}(\phi_{n})<t_{n}^{m}<t^{+}(\phi_{n}^{m})$ and with $\frac{D(\phi_{n}%
				^{m}(t_{n}^{m}))}{1+t_{n}^{m}}<g^{+}(x_{n})$.
		\end{itemize}
		
		For each $n\in\mathbb{N}$, we denote $y_{n}=x_{n}^{m_{n}}$, $t_{n}%
		=t_{n}^{m_{n}}$, $\psi_{n}=\phi_{n}^{m_{n}}$, for some $m_{n}\in\mathbb{N}$,
		such that
		\[
		d(y_{n},x_{n})<2^{-n},\mbox{ }d(\psi_{n}(t^{+}(\phi_{n})),\partial U)<2^{-n},
		\]
		$t^{+}(\phi_{n})<t_{n}<t^{+}(\psi_{n})$ and with
		\begin{equation}
			\frac{D(\psi_{n}(t_{n}))}{1+t_{n}}<g^{+}(x_{n}). \label{eq:des.tn.g}%
		\end{equation}
		This is possible because $\phi_{n}(t^{+}(\phi_{n}))\in\partial U$,
		$n\in\mathbb{N}$, and $\phi_{n}^{m}(t^{+}(\phi_{n}))\rightarrow\phi_{n}%
		(t^{+}(\phi_{n}))$ as $m$ goes to $+\infty$.
		
		We have two possibilities: $t^{+}(\phi_{n})\rightarrow+\infty$ as
		$n\rightarrow+\infty$ or $\{t^{+}(\phi_{n})\}_{n\in\mathbb{N}}$ is bounded.
		
		First, suppose that $t^{+}(\phi_{n})\rightarrow+\infty$ as $n\rightarrow
		+\infty$.

		For all $n\in\mathbb{N}$, set $s_{n}=\frac{t^{+}(\phi_{n})}{2}$, and we have
		$\phi_{n}(s_{n})\in N$ and $s_{n}\rightarrow+\infty$ as $n$ goes to $+\infty$.
		By the admissibility of $N$ and Proposition \ref{RybTheo4.5}, we may assume
		that $\phi_{n}(s_{n})$ converges to some $y_{0}\in A^{-}(N)$. On the other
		hand, by {(K4)}, we may assume that $\eta_{n}\in\mathcal{R}$ given by
		$\eta_{n}(\cdot)=\phi_{n}(s_{n}+\cdot)$, $n\in\mathbb{N}$, converges to some
		$\eta\in\mathcal{R}$ with $\eta(0)=y_{0}$. Arguing as above, since $t^{+}%
		(\eta_{n})=s_{n}\rightarrow+\infty$, we have $s^{+}(\eta
			)=+\infty$ and $\eta(\mathbb{R}^{+})\subset N$. Thus $y_{0}\in A^{+}(N)$,
		which implies that $y_{0}\in K$.
		
		By \eqref{eq:des.tn.g} and the definition of $g^{+}$ we have, for each
		$n\in\mathbb{N}$,
		\begin{equation}
			D(\psi_{n}(t_{n}))<\frac{(1+t_{n})D(\phi(s_{n}))}{1+s_{n}}. \label{IneqDFin}%
		\end{equation}
		Observe that our choice of $\psi_{n}$ implies that $\psi_{n}([0,t^{+}(\phi
		_{n})])\subset N$, for all $n\in\mathbb{N}$, and, since $N$ is admissible, we
		have $\psi_{n}(t^{+}(\phi_{n}))\rightarrow z_{0}\in N$, and $z_{0}\in
		A^{-}(N)$, by Proposition \ref{RybTheo4.5}.
		
		We have two other possibilities: either $\{t_{n}-t^{+}(\phi_{n})\}_{n\in
			\mathbb{N}}$ is bounded or not.
		
		\begin{itemize}
			\item Suppose that $\{t_{n}-t^{+}(\phi_{n})\}_{n\in\mathbb{N}}$ is bounded
			and, without loss of generality, we may assume that as $n$ goes to $+\infty$,
			we have $t_{n}-t^{+}(\phi_{n})\rightarrow\tau_{0}$, for some $\tau_{0}%
			\in\mathbb{R}^{+}$.
			
			In that case we have that
			\[
			\frac{1+t_{n}}{1+s_{n}}=\frac{(1+t^{+}(\phi_{n}))+t_{n}-t^{+}(\phi_{n}%
				)}{1+\tfrac{t^{+}(\phi_{n})}{2}}%
			\]
			is uniformly bounded for $n\in\mathbb{N}$. Observe that $D(\phi_{n}%
			(s_{n}))\rightarrow D(y_{0})=0$ as $n\rightarrow+\infty$ and, then
				\eqref{IneqDFin} implies
			\begin{equation}
				D(\psi_{n}(t_{n}))\rightarrow0\mbox{ as }n\rightarrow+\infty.
				\label{eq:Dpsi.tn}%
			\end{equation}

			For each $n\in\mathbb{N}$, define $\varphi_{n}:\mathbb{R}^{+}\rightarrow X$
			given by $\varphi_{n}(\cdot)=\psi_{n}(t^{+}(\phi_{n})+\cdot)$, which belongs
			to $\mathcal{R}$, by {(K2)}. By {(K4)}, we may assume that
				$\varphi_{n}\rightarrow\varphi$ uniformly on compacts of $\mathbb{R}^{+}$, for
				some $\varphi\in\mathcal{R}$, with $\varphi(0)=z_{0}$. As a consequence,
				$\psi_{n}(t_{n})=\varphi_{n}(t_{n}-t^{+}(\phi_{n}))\rightarrow\varphi(\tau
				_{0})$ as $n$ goes to $+\infty$.  Using that $D$ is continuous and
			\eqref{eq:Dpsi.tn}, we have $D(\varphi(\tau_{0}))=0$ and, consequently,
			$\varphi(\tau_{0})\in K$. Also, $\varphi([0,\tau_{0}])\subset N$, by
			$\varphi_{n}([0,t_{n}-t^{+}(\phi_{n})])\subset N$ and by the argument in
			Proposition \ref{RybTheo4.5}. It is easy to see that $z_{0}\in A^{+}(N)$,
			using property (K3).
			
			Therefore, $z_{0}\in K\cap\partial U=\emptyset$, which is a contradiction.
			
			\item If $\{t_{n}-t^{+}(\phi_{n})\}_{n\in\mathbb{N}}$ is unbounded, we may
			assume that $t_{n}-t^{+}(\phi_{n})\rightarrow+\infty$ as $n\rightarrow+\infty$.
			
			For each $n\in\mathbb{N}$, define $\varphi_{n}\in\mathcal{R}$ as above, hence
			$\varphi_{n}([0,t_{n}-t^{+}(\phi_{n})])\subset N$. Again, by (K4), we may
			assume that $\varphi_{n}\rightarrow\varphi$ uniformly on compacts of
			$\mathbb{R}^{+}$, for some $\varphi\in\mathcal{R}$, with $\varphi(0)=z_{0}$.
			By the arguments in Proposition \ref{RybTheo4.5}, we have that $\varphi
			(\mathbb{R}^{+})\subset N$. Therefore, $z_{0}\in A^{+}(N)$ and we arrive at a
			contradiction since $z_{0}\in A^{-}(N)\cap\partial U$.
		\end{itemize}
		
		Thus, the only remaining possibility is that the sequence $\{t^{+}(\phi
		_{n})\}_{n\in\mathbb{N}}$ is bounded. Therefore, without loss of generality,
		we may assume that $t^{+}(\phi_{n})\rightarrow t_{0}$, for some $t_{0}%
		\in\mathbb{R}^{+}$.
		
		We can also prove that $\{t^{+}(\psi_{n})\}_{n\in\mathbb{N}}$ is bounded. In
		fact, if this is not true, by (K4), we could assume that $\psi_{n}%
		\rightarrow\psi$ uniformly on compacts of $\mathbb{R}^{+}$, for $\psi
		\in\mathcal{R}$, $\psi(0)=x_{0}$ and $s^{+}(\psi)=+\infty$, by
		the argument in Proposition \ref{RybTheo4.5}. Since $x_{0}\in K$,
		$\psi([0,+\infty))\subset N$, and $K$ is the largest weakly invariant set in
		this neighborhood, we conclude, using (K3), that $\psi([0,+\infty))\subset K$.
		On the other hand, by the choice of $y_{n}\in X$, $n\in\mathbb{N}$, and the
		uniform convergence of $\psi_{n}$ to $\psi$, we find that $\psi(t_{0}%
		)\in\partial U.$ This is a contradiction.
		
		Therefore, $\{t^{+}(\psi_{n})\}_{n\in\mathbb{N}}$ is also bounded and it may
		be assumed to be convergent to a point $t_{1}\in\mathbb{R}^{+}$. Since
		$t^{+}(\phi_{n})<t^{+}(\psi_{n})$, for all $n\in\mathbb{N}$, we have
		$t_{0}\leq t_{1}$.{ As a consequence, $\{t_{n}\}_{n\in\mathbb{N}}$ is bounded
			and we may assume $t_{n}\rightarrow\tau$, for some $\tau\in\lbrack t_{0}%
			,t_{1}]$. By \eqref{eq:x_nK}, we have that $g^{+}(x_{n})\rightarrow0$ which,
			together with \eqref{eq:des.tn.g}, implies ${D(\psi(\tau))}=0$. }
		%

		We will show that $\psi(t_{0})\in A^{+}(N)\cap A^{-}(N)=K$, but $\psi
		(t_{0})\in\partial U$ and that leads us to a contradiction. It is easy to see
		that $\psi(t_{0})\in A^{-}(N)$, since $\psi(0)=x_{0}\in K\subset A^{-}(N)$ and
		$\psi([0,t_{0}])\subset N$.
		Also $\psi(\tau)\in K\subset A^{+}(N)$, which means that we find $\psi_{1}%
		\in\mathcal{R}$, $\psi_{1}(0)=\psi(\tau)$, and $\psi_{1}(\mathbb{R}%
		^{+})\subset K$. Hence, by {(K2)} and {(K3)}, the map $\xi:\mathbb{R}%
		^{+}\rightarrow X$ given by
		\[
		\xi(t)=%
		\begin{cases}
			\psi(t_{0}+t),\ \mbox{ if }t\in\lbrack0,\tau-t_{0}]\\
			\psi_{1}(t-\tau+t_{0})\mbox{ if }t\geq\tau-t_{0}%
		\end{cases}
		\]
		belongs to $\mathcal{R}$, $\xi(0)=\psi(t_{0})$ and $\xi(\mathbb{R}^{+})\subset
		N$. Consequently, $\psi(t_{0})\in A^{+}(N)$. So, we have a contradiction.
		
		Therefore, there exists an open neighborhood $W$ of $K$ with $W\subset
		\mathcal{O}(K)\cap U$ for which the restriction of $g^{+}$ to $W$ is also
		lower semicontinuous.
		
		\item[3)] Suppose, by contradiction, that $g^{-}$ is not upper semicontinuous
		in $N$. Then, for some $x_{0}\in N$, we could find $\mu>0$ and a sequence
		$\{x_{n}\}_{n\in\mathbb{N}}\subset N$ with $x_{n}\rightarrow x_{0}$ and
		$g^{-}(x_{0})<\mu<g^{-}(x_{n})$, for all $n\in\mathbb{N}$.
		
		Then, for each $n\in\mathbb{N}$, there are $\phi_{n}\in\mathcal{N}$, $\phi
		_{n}(0)=x_{n}$, with $f^{-}(\phi_{n})>\mu$. We also find $\{t_{n}%
		\}_{n\in\mathbb{N}}\in\mathbb{R}^{+}$ such that $t_{n}\in\lbrack0,s^{+}%
		(\phi_{n}))$ and $\alpha(t_{n})F(\phi_{n}(t_{n}))>\mu$, for all $n\in
		\mathbb{N}$. We may assume, w.l.g., that $\phi_{n}(t)\rightarrow\phi(t)$, for
		$t$ on compacts of $\mathbb{R}^{+}$, for some $\phi\in\mathcal{N}$ with
		$\phi(0)={x_{0}}$.
		
		There are two possibilities: $\{t_{n}\}_{n\in\mathbb{N}}$ is bounded or
		$t_{n}\rightarrow+\infty$.
		
		In the first case, we may assume that $t_{n}$ converges to some $t_{0}%
		\in\mathbb{R}^{+}$ as $n$ goes to $+\infty$. Thus, $\mu\leq\alpha(t_{0}%
		)F(\phi(t_{0}))$. By the hypothesis, it follows that $\phi(t_{0})\notin N$,
		which cannot happen since $N$ is closed and $\{\phi_{n}(t_{n})\}_{n\in
			\mathbb{N}}\in N$ converges to $\phi(t_{0})$ as $n\rightarrow+\infty$.
		
		Now, if we assume that $t_{n}\rightarrow+\infty$, then $s^{+}(\phi)=+\infty$.
		Taking a subsequence if necessary, we have $\phi_{n}(t_{n})\rightarrow y$, for
		some $y\in A^{-}(N)$, by Proposition \ref{RybTheo4.5}. As a consequence,
		$\alpha(t_{n})F(\phi_{n}(t_{n}))\rightarrow0$ as $n\rightarrow+\infty$. But
		this is a contradiction with the hypothesis $\mu<\alpha(t_{n})F(\phi_{n}%
		(t_{n}))$, for all $n\in\mathbb{N}$.
		
		Therefore, we have shown the upper semicontinuity of $g^{-}$ in $N$.
		
		Now, we want to show that $g^{-}$ is lower semicontinuous in any neighborhood
		$W$ of $K$ in $\mathcal{O}(K)\cap U$ for which $t^{+}(\phi)=s^{+}(\phi)$, for
		all $\phi\in\mathcal{U}$ with $\phi(0)\in W$. Consider any neighborhood $W$
		satisfying the required conditions. Assume, by contradiction, that we can find
		$\mu>0$, $x\in W$ and a sequence $\{x_{n}\}_{n\in\mathbb{N}}\in W$ with
		$x_{n}\rightarrow x$ as $n\rightarrow+\infty$ and
		\begin{equation}
			g^{-}(x_{n})\leq\mu<g^{-}(x). \label{eq:ls.gminus}%
		\end{equation}

		Hence, there is $\phi\in\mathcal{R}$, $\phi(0)=x$ such that $\mu<f^{-}(\phi)$.
		By definition of $f^{-}$ and by $t^{+}(\phi)=s^{+}(\phi)$, we find $\tau
		\in(0,s^{+}(\phi))$ such that $\phi([0,\tau])\subset U$ and $\alpha
		(\tau)F(\phi(\tau))>\mu$. Using {(K5)}, we may assume that there exists a
		sequence $\phi_{n}\in\mathcal{R}$, $\phi_{n}(0)=x_{n}$ such that $\phi
		_{n}\rightarrow\phi$ uniformly on compacts of $\mathbb{R}^{+}$. The continuity
		of $F$ assures that we can find $n_{0}\in\mathbb{N}$, such that, for all
		$n\geq n_{0}$, $\phi_{n}([0,\tau])\subset U$ and
		\[
		\alpha(\tau)F(\phi_{n}(\tau))>\mu.
		\]
		On the other hand, by \eqref{eq:ls.gminus} it follows that
		\[
		\mu<\alpha(\tau)F(\phi_{n}(\tau))\leq f^{-}(\phi_{n})\leq g^{-}(x_{n})\leq
		\mu,
		\]
		which is a contradiction.
	\end{itemize}
\end{proof}

\begin{lemma}
	\label{lemma:convK} Assume (K1)-(K4). Let $K\neq\emptyset$ be a closed
	isolated weakly invariant set. Suppose that $N$ is a closed $G$-admissible
	isolating neighborhood of $K$. Assume that we have a sequence $\{x_{n}%
	\}_{n\in\mathbb{N}}\in U$ such that $g^{+}(x_{n})\rightarrow0$ and
	$g^{-}(x_{n})\rightarrow0$ as $n\rightarrow+\infty$. Then we find a
	subsequence $\{x_{n_{m}}\}_{m\in\mathbb{N}}$ and $x\in K$ such that $x_{n_{m}%
	}\rightarrow x$ as $m\rightarrow+\infty$.
\end{lemma}

\begin{proof}
	By the definition of $g^{+}$, we can find $\phi_{n} \in\mathcal{U}_{x_{n}}$,
	such that $f^{+}(\phi_{n})\rightarrow0$ as $n \rightarrow+\infty$.
	
	Now, by definition of $g^{-}$, $f^{-}(\phi_{n})\rightarrow0$ and,
	consequently, $d(x_{n}, A^{-}(N))=F(x_{n})\rightarrow0$ as $n \rightarrow
	+\infty$. Since $A^{-}(N)$ is compact, by Proposition
		\ref{RybTheo4.5}, we may assume that $x_{n}\rightarrow x$, for some $x \in
	A^{-}(N)$.
	
	We have two possibilities: either $\{t^{+}(\phi_{n})\}_{n\in\mathbb{N}}$ is
	bounded or it is unbounded.
	
	\begin{itemize}
		\item[i)] Suppose that $t^{+}(\phi_{n})<M$ for all $n\in\mathbb{N}$, and some
		$M>0$. Then, for each $n\in\mathbb{N}$, there is $t_{n}\in\lbrack0,t^{+}%
		(\phi_{n})]$ such that
		\begin{equation}
			f^{+}(\phi_{n})\geq\frac{\inf\{D(\phi_{n}(t)):t\in\lbrack0,t^{+}(\phi_{n}%
				)]\}}{1+M}=\frac{D(\phi_{n}(t_{n}))}{1+M}. \label{eq:fandM}%
		\end{equation}

		Without loss of generality $t_{n}\rightarrow t_{0}<\infty$, for some $t_{0}%
		\in\mathbb{R}^{+}$. By {(K4)} and Proposition \ref{RybTheo4.5}, there is
		$\phi\in\mathcal{N}_{x}$ with $\phi([0,t_{0}])\subset N$, for which $\phi
		_{n}\rightarrow\phi$ uniformly on compacts of $\mathbb{R}^{+}$. Since
		$D(\phi_{n}(t_{n}))\rightarrow D(\phi(t_{0}))$ as $n\rightarrow+\infty$ by
		\eqref{eq:fandM}, we obtain $D(\phi(t_{0}))=0$ and $\phi(t_{0})\in K$.
		
		In particular, it follows that $x \in A^{+}(N)$. Thus $x \in K$.
		
		\item[ii)] If $\{t^{+}(\phi_{n})\}_{n\in\mathbb{N}}$ is unbounded. We may
		assume that $t^{+}(\phi_{n})\rightarrow+\infty$ and, by point (a2) in
		Proposition \ref{RybTheo4.5}, we have $x\in A^{+}(N)$. Therefore, $x\in
		A^{-}(N)\cap A^{+}(N)=K$, as desired.
	\end{itemize}
\end{proof}

\begin{lemma}
	\label{lemma:Hepsilon} Let $K\neq\emptyset$ be a closed isolated weakly
	invariant set. {Assume (K1)-(K5). }Suppose that $N$ is a closed $G$-admissible
	isolating neighborhood of $K$. Consider $\varepsilon>0$ and the set
	\[
	H_{\varepsilon}=\{x\in U\cap\mathcal{O}(K):g^{+}(x)<\varepsilon,\ g^{-}%
	(x)<\varepsilon\}.
	\]

	Then $H_{\varepsilon}$ is open in $U\cap\mathcal{O}(K)$. We can choose
	$\varepsilon>0$ small enough such that $clH_{\varepsilon}\subset
		U\cap O(K)$, $g^{+}$ is continuous on $clH_{\varepsilon}$ and
	$clH_{\varepsilon}$ is an isolating neighborhood of $K$.
\end{lemma}

\begin{proof}
	Since, by Proposition \ref{prop:properties.g}, both $g^{+}$ and $g^{-}$ are
	upper semicontinuous in $U\cap\mathcal{O}(K)$, $H_{\varepsilon}$ is open for
	every $\varepsilon>0$. Observe that $K\subset H_{\varepsilon}$ since
	$g^{+}(x)=g^{-}(x)=0$, for all $x\in K$.
	
	Consider $W\subset U\cap\mathcal{O}(K)$, an open neighborhood of $K$ for which
	$g^{+}$ is continuous in $W$, whose existence is assured by Proposition
	\ref{prop:properties.g}. We want to show that there is
		$\varepsilon>0$ such that $clH_{\varepsilon}\subset W$. If this was not true,
		then we would find sequences $\{\varepsilon_{n}\}_{n\in\mathbb{N}}%
		\in\mathbb{R}^{+}$, $\{y_{n}\}_{n\in\mathbb{N}}\in X$, with $y_{n}\in
		clH_{\varepsilon_{n}}\setminus W$, for all $n\in\mathbb{N}$, and
		$\varepsilon_{n}\rightarrow0$ as $n\rightarrow+\infty$. For each
		$n\in\mathbb{N}$, take $x_{n}\in H_{\varepsilon_{n}}$ with $d(x_{n}%
		,y_{n})<\varepsilon_{n}$. It follows that {\ $g^{+}(x_{n})$ and $g^{-}(x_{n})$
			go to $0$ as $n\rightarrow+\infty$.} By Lemma \ref{lemma:convK}, we may assume
		that $x_{n}\rightarrow x\in K$ and then $y_{n}\rightarrow x$. We thus obtain
		that $x\in K$ but $x\not \in W$, which is a contradiction since $K\subset W$.
		Therefore, $g^{+}$ is continuous in $clH_{\varepsilon}$. 
		
		Finally, as $K\subset H_{\varepsilon}\subset N$, the largest weakly invariant
		subset of $H_{\varepsilon}$ contains $K$ and it must be inside $N$, hence it
		must be $K$.
\end{proof}

\begin{theorem}
	\label{theo:Isol.Block.(K1)-(K5)} Let $K\neq\emptyset$ be a closed isolated
	weakly invariant set. {Suppose that $\mathcal{R}$ satisfies (K1)-(K5) and that
		there is a closed isolating neighborhood $N$ of $K$ which is $G$-admissible.
	}Then there exists an isolating block $B$ with $K\subset B\subset N$.
\end{theorem}

\begin{proof}
	Choose $\varepsilon>0$, the number provided in Lemma
		\ref{lemma:Hepsilon}, and put $\tilde{U}=H_{\varepsilon}$ and $\tilde
		{N}=clH_{\varepsilon}\subset U\cap\mathcal{O}(K)$.
	
	Define the functions $\tilde{t}^{+}$, $\tilde{s}^{+}$ and $\tilde{g}^{+},
	\tilde{g}^{-}$ as before, with $U$ (resp. $N$) replaced by $\tilde{U}$ (resp.
	$\tilde{N} $).
	
	Observe that all the previous results can be applied to the functions defined
	above. It is also easy to see that $\tilde{N}$ is admissible.
	
	We want to show that $\tilde{t}^{+}(\phi)=\tilde{s}^{+}(\phi)$, for every
	$\phi\in\mathcal{R}$, with $\phi(0)\in\tilde{U}$. Clearly, $\tilde{t}^{+}%
	(\phi)\leq\tilde{s}^{+}(\phi)$, for every $\phi\in\mathcal{R}$, with
	$\phi(0)\in\tilde{U}$. Suppose, by contradiction, that we can find $x
	\in\tilde{U}$ and $\psi\in\mathcal{R}$ with $\psi(0)=x$ and such that
	$\tilde{t}^{+}(\psi)<\tilde{s}^{+}(\psi)$.
	
	We have $y=\psi(\tilde{t}^{+}(\psi))\in\partial\tilde{U}\subset U\cap
	\mathcal{O}(K)$. Hence, either $g^{+}(y)\geq\varepsilon$ or $g^{-}%
	(y)\geq\varepsilon$. Since $x\in\tilde{U}$ and $g^{-}(x)\geq g^{-}(y)$, by
	Proposition \ref{prop:monot.gminus}, it follows that the last scenario above
	cannot happen. So, necessarily $g^{+}(y)\geq\varepsilon$. Now as $y\in
	\partial\tilde{U}$, we find a sequence $\{y_{n}\}_{n\in\mathbb{N}}\in\tilde
	{U}$ with $y_{n}\rightarrow y$ as $n\rightarrow+\infty$. Hence, $g^{+}%
	(y_{n})<\varepsilon$ and then, by the continuity of $g^{+}$ in
		$clH_{\varepsilon}$, we have $g^{+}(y)=\varepsilon$.
	
	Choose $t\in(\tilde{t}^{+}(\psi),\tilde{s}^{+}(\psi)),$ so $\phi
	([0,t])\subset\tilde{N}$. Then, we have $g^{+}(\phi(t))\leq\varepsilon$. On
	the other hand, the strict inequality property of $g^{+}$ along orbits in $U$
	implies that $g^{+}(\phi(t))>g^{+}(y)=\varepsilon$, which is a contradiction.
	
	Therefore, $\tilde{t}^{+}(\phi)=\tilde{s}^{+}(\phi)$, for all $\phi
	\in\mathcal{R}$, with $\phi(0)\in\tilde{U}$. Consequently, by Proposition
	\ref{prop:properties.g}, we conclude that $\tilde{g}^{-}$ is continuous in
	$\tilde{U}$.
	
	For $\delta\in(0,\varepsilon)$, define
	\[
	B=cl\tilde{H}_{\delta}=cl\{x\in\tilde{U}:\tilde{g}^{+}(x)<\delta,\ \tilde
	{g}^{-}(x)<\delta\}.
	\]

	Applying Lemma \ref{lemma:Hepsilon} to {$\tilde{U}$ and
			$\tilde{N}$ we obtain }$\delta<\varepsilon$ such that: 
		
		\begin{itemize}
			\item $\tilde{H}_{\delta}$ is open in $U\cap\mathcal{O}(K)$, $K\subset
			\tilde{H}_{\delta}$ and $cl\tilde{H}_{\delta}\subset U\cap\mathcal{O}(K);$
			
			\item $\tilde{g}^{+},\tilde{g}^{-}$ are continuous on $cl\tilde{H}_{\delta}.$
		\end{itemize}
		
		Let us show that $B\subset H_{\varepsilon}$. In fact, if $x\in B$, then there
		is a sequence $\{x_{n}\}_{n\in\mathbb{N}}\in H_{\delta}$ with $x_{n}%
		\rightarrow x$, as $n\rightarrow\infty$. Hence $\tilde{g}^{+}(x_{n})<\delta$
		and $\tilde{g}^{-}(x_{n})<\delta$. The continuity of $\tilde{g}^{+}$ and
		$\tilde{g}^{+}$ imply that $\tilde{g}^{+}(x),\tilde{g}^{-}(x)\leq
		\delta<\varepsilon$.
	
	Now, observe that $\partial B=b^{-}\cup b^{+}\cup b^{\star}$, where
	\[
	\begin{aligned} b^- = \{x \in \partial B : \tilde{g}^+(x)=\delta, \ \tilde{g}^-(x)<\delta\} ,\\ b^+ = \{x \in \partial B : \tilde{g}^+(x)<\delta, \ \tilde{g}^-(x)=\delta\}, \\ b^{\star} = \{x \in \partial B : \tilde{g}^+(x)=\delta, \ \tilde{g}^-(x)=\delta\}. \end{aligned}
	\]

	Consider a point $x\in\partial B$ and a function $\phi:[-\gamma_{1}%
	,+\infty)\rightarrow X$ such that $\phi(\cdot-\gamma_{1}%
		)\in\mathcal{R}$ with $\phi(0)=x$, where $\gamma_{1}\geq0$. We can choose
		constants $\tau_{1}\geq0$ and $\tau_{2}>0$ such that $\phi([-\tau_{1},\tau
		_{2}])\subset\tilde{U}$.
	
	\begin{itemize}
		\item Suppose that $x\in b^{-}$. By definition of $b^{-}$ and the monotonicity
		of $\tilde{g}^{+}$ and $\tilde{g}^{-}$ along orbits we have, for $t\in
		(0,\tau_{2}]$,
		\[
		\tilde{g}^{+}(\phi(t))>\tilde{g}^{+}(x)=\delta\mbox{ and }\tilde{g}^{-}%
		(\phi(t))\leq\tilde{g}^{-}(x)<\delta.
		\]
		Then $\phi((0,\tau_{2}])\subset X\setminus B$.
		
		On the other hand, if $\gamma_{1}>0$, then $\tau_{1}$ can be
			chosen positive as well. Since $\tilde{g}^{+}$ is continuous and $\phi
			(\cdot-{\tau_{1}})\in\mathcal{R}$, we find $\sigma_{1}\in(0,\tau_{1})$ such
		that $\tilde{g}^{+}(\phi(\sigma_{1}))\neq0$ and, by the monotonicity of
		$\tilde{g}^{+}$, we have that $\tilde{g}^{+}(\phi(t))<\tilde{g}^{+}(x)=\delta
		$, for all $t\in\lbrack-\sigma_{1},0)$. The continuity of $\tilde{g}^{-}$
		assures that there is a $\tau\in\lbrack-\sigma_{1},0)$ for which $\tilde
		{g}^{-}(\phi(t))<\delta$, for all $t\in\lbrack\tau,0)$.
		
		Hence, $\phi([\tau,0))\subset intB$.
		
		Therefore, each point of $b^{-}$ is an egress point, see Definition
		\ref{def:point_eiboff}.
		
		\item Suppose that $x\in b^{+}$. By the monotonicity of $\tilde{g}^{-}$ along
		orbits in $\tilde{U}$, we have, for $t\in(0,\tau_{2}]$,
		\[
		\tilde{g}^{-}(\phi(t))<\tilde{g}^{-}(x)=\delta,
		\]
		and, since $\tilde{g}^{+}(x)<\delta$, by the continuity of $\tilde{g}^{+}$, we
		find $\tau\in(0,\tau_{2}]$ such that\\ $\tilde{g}^{+}(\phi(t))<\delta$ for
		$t\in(0,\tau]$. Hence, $\phi((0,\tau])\subset intB$.
		
		Also, if $\tau_{1}>0$, by the monotonicity of $\tilde{g}^{-}$ and $\tilde
		{g}^{+}$ along orbits for $t\in\lbrack-\tau_{1},0)$, we have $\tilde{g}%
		^{+}(\phi(t))\leq\tilde{g}^{+}(x)<\delta$ and $\tilde{g}^{-}(\phi
		(t))>\tilde{g}^{-}(x)=\delta$. Then $\phi([-\tau_{1},0))\subset X\setminus B$.
		
		That means each point of $b^{+}$ is an ingress point.
		
		\item Suppose that $x\in b^{\star}$. By the monotonicity of $\tilde{g}^{+}$
		and $\tilde{g}^{-}$ along orbits in $\tilde{U}$, we have, for all $t\in
		(0,\tau_{2}]$, $\delta=\tilde{g}^{+}(x)<\tilde{g}^{+}(\phi(t))$ and
		$\delta=\tilde{g}^{-}(x)>\tilde{g}^{-}(\phi(t))$.
		
		Also, if $\tau_{1}\neq0$, $\delta=\tilde{g}^{+}(x)\geq\tilde
			{g}^{+}(\phi(t))$ and $\delta=\tilde{g}^{-}(x)<\tilde{g}^{-}(\phi(t))$, for
		all $t\in\lbrack-\tau_{1},0)$.
		
		Thus $\phi(t)\in X\setminus B$, for all $t\in\lbrack-\tau_{1},0)\cup
		(0,\tau_{2}]$.
		
		That means each point of $b^{\star}$ is a bounce-off point.
	\end{itemize}
	
	Therefore, $B$ is an isolating block since $B^{-}=b^{-}\cup b^{\star}$ is closed.
\end{proof}

\begin{theorem}
	\label{theo:exist_block_subsemiflow} Let $\widetilde{\mathcal{R}}%
	\supset\mathcal{R}$ be sets of functions satisfying $\left(  K1\right)
	-\left(  K4\right)  $ and let $\tilde{G}\supset G$ be their associated
	multivalued semiflows. Assume that $K$ is a closed isolated weakly invariant
	set for $G$ with the closed isolating $G$-admissible neighborhood $N$. Also,
	let $\widetilde{K}$ be a closed isolated weakly invariant set for
	$\widetilde{G}$ such that $K\subset\widetilde{K}$ and $N$ is an isolating
	$\widetilde{G}$-admissible neighborhood for $\widetilde{K}$ as well. Moreover,
	we suppose that $\widetilde{\mathcal{R}}$ satisfies $\left(  K5\right)  $ for
	$\widetilde{K}$. Then there is an isolating block $B$ for $K$.
\end{theorem}

\begin{proof}
	Since $\widetilde{\mathcal{R}}$ satisfies (K1)-(K5) for $\widetilde{K}$, the
	set $M=\overline{N\cap O(\widetilde{K})}$ (where $O(\widetilde{K})$ is the
	neighborhood from condition $\left(  K5\right)  $) is a closed isolating
	admissible neighborhood of $\widetilde{K}$. Hence, by Theorem
	\ref{theo:Isol.Block.(K1)-(K5)}, $\widetilde{K}$ has an isolating block $B$
	for $\tilde{G}$. Since any $\varphi\in\mathcal{R}$ belongs to $\tilde
	{\mathcal{R}}$, $B$ is also an isolating block of $K$ for $G$.
\end{proof}

\section{Examples for condition $(K5)$}

In this section we will see one example and one counterexample that are useful
in order to clarify the ideas about the property (K5).

\subsection{An ordinary differential equation without uniqueness}

Let us consider the equation%
\begin{equation}
	x^{\prime}=\sqrt{\left\vert x\right\vert }. \label{eq2}%
\end{equation}
The phase space is $\mathbb{R}$. For $x\left(  0\right)  =x_{0}>0$ the unique
solution is
\[
x\left(  t\right)  =\left(  \frac{t}{2}+\sqrt{x_{0}}\right)  ^{2},
\]
whereas for $x_{0}<0$ the unique solution is
\[
x\left(  t\right)  =-\left(  -\frac{t}{2}+\sqrt{-x_{0}}\right)  ^{2}.
\]

For $x\left(  0\right)  =0$ we have infinite solutions given by%
\[
\overline{x}\left(  t\right)  \equiv0,
\]%
\[
x_{\tau}^{+}\left(  t\right)  =\left\{
\begin{array}
	[c]{c}%
	0\text{, }0\leq t\leq\tau,\\
	\frac{\left(  t-\tau\right)  ^{2}}{4}\text{, }t\geq\tau,
\end{array}
\right.
\]%
\[
x_{\tau}^{-}\left(  t\right)  =\left\{
\begin{array}
	[c]{c}%
	0\text{, }0\leq t\leq\tau,\\
	-\frac{\left(  t-\tau\right)  ^{2}}{4}\text{, }t\geq\tau,
\end{array}
\right.
\]
for all $\tau\geq0$.We observe that $x_{0}^{+}\left(  t\right)  $ is the
maximal solution for $x_{0}=0$ and $x_{0}^{-}\left(  t\right)  $ is the
minimal solution for $x_{0}=0$. Also,
\[
x\left(  t\right)  =\left(  \frac{t}{2}+\sqrt{x_{0}}\right)  ^{2}\rightarrow
x_{0}\left(  t\right)  =\frac{t^{2}}{4},\ \text{if }x_{0}\rightarrow0^{+},
\]%
\[
x\left(  t\right)  =-\left(  -\frac{t}{2}+\sqrt{-x_{0}}\right)  ^{2}%
\rightarrow x_{0}\left(  t\right)  =-\frac{t^{2}}{4},\ \text{if }%
x_{0}\rightarrow0^{-}.
\]
Thus, when $x_{0}\rightarrow0^{+}$ the unique solution converges to the
maximal solution for $0$ and when $x_{0}\rightarrow0^{-}$ the unique solution
converges to the minimal solution for $x_{0}=0$.

Therefore, condition $\left(  K5\right)  $ is not satisfied in any
neighborhood of the fixed point $0$.

\subsection{A differential inclusion with Lipschitz nonlinearity}

For a Banach space $X$ let $C_{v}(X)$ be the set of all non-empty, bounded,
closed, convex subsets of $X.$

Let us consider the boundary-value problem
\begin{equation}
	\left\{
	\begin{array}
		[c]{c}%
		\dfrac{\partial u}{\partial t}-\Delta u\in f(u)+q,\text{ on }\Omega
		\times(0,T),\\
		u=0\text{, on }\partial\Omega\times(0,T),\\
		u(x,0)=u_{0}(x)\text{ on }\Omega,
	\end{array}
	\right.  \label{Inclusion2}%
\end{equation}
where $\Omega\subset\mathbb{R}^{n}$ is an open bounded set with smooth
boundary and {$q\in L^{2}(\Omega)$}. We assume that the multivalued map $f$
satisfies the following assumptions:

\begin{itemize}
	\item[$\left(  f1\right)  $] $f:\mathbb{R}\rightarrow C_{v}(\mathbb{R})$.
	
	\item[$\left(  f2\right)  $] $f$ is Lipschitz in the multivalued sense, i.e.
	there is $C\geq0$ such that
	\begin{equation}
		dist_{H}(f(x),f(z))\leq C\left\vert x-z\right\vert ,\ \forall x,z\in
		\mathbb{R}. \label{fLip}%
	\end{equation}
	
\end{itemize}

Let us define the multivalued map $F:D(F)\subset L^{2}(\Omega)\rightarrow
P(L^{2}(\Omega))$ given by%
\begin{equation}
	F(y(\cdot))=\{\xi(\cdot)\in L^{2}(\Omega):\xi=\widetilde{\xi}
	+q,\ \widetilde{\xi}(x)\in f(y(x))\text{ a.e. on }\Omega\}. \label{F}%
\end{equation}
It is known \cite[Lemmas 11, 12]{MelnikValero98} that:

\begin{itemize}
	\item[$\left(  F1\right)  $] $F:L^{2}(\Omega)\rightarrow C_{v}(L^{2}(\Omega))$;
	
	\item[$\left(  F3\right)  $] $F$ is Lipschitz with the same Lipschitz constant
	as $f$, that is,
	\[
	dist_{H}(F(u),F(v))\leq C\left\Vert u-v\right\Vert _{L^{2}},\ \forall u,v\in
	L^{2}\left(  \Omega\right)  .
	\]
	
\end{itemize}

The operator $A=-\Delta:H^{2}(\Omega)\cap H_{0}^{1}(\Omega)\rightarrow
L^{2}(\Omega)$ is maximal monotone in $L^{2}\left(  \Omega\right)  $. Hence,
inclusion \eqref{Inclusion2} can be written in the abstract form%
\begin{equation}
	\left\{
	\begin{array}
		[c]{c}%
		\dfrac{du}{dt}+Au\in F(u),\text{ }t>0,\\
		u(0)=u_{0}\in L^{2}\left(  \Omega\right)  ,
	\end{array}
	\right.  \label{Inclusion3}%
\end{equation}

If we assume additionally the existence of $M\geq0,\,\epsilon>0$ such that
\begin{equation}
	zs\leq(\lambda_{1}-\epsilon)\left\vert s\right\vert ^{2}+M,\ \forall
	s\in\mathbb{R},\forall z\in f(s), \label{LCondition3}%
\end{equation}
where $\lambda_{1}$ is the first eigenvalue of $-\Delta$ in $H_{0}^{1}%
(\Omega)$, then this problem generates a strict multivalued semiflow in
$L^{2}\left(  \Omega\right)  $ having a global compact attractor $\mathcal{A}$
\cite{MelnikValero98}.

Let us check that the solutions of \eqref{Inclusion3} satisfy property (K5).

The function $u\in C([0,+\infty),L^{2}\left(  \Omega\right)  )$ is a strong
solution of problem \eqref{Inclusion3} if there is a selection $h\in
L_{loc}^{2}(0,+\infty;L^{2}\left(  \Omega\right)  )$, $h\left(  t\right)  \in
F(u(t))$ for a.a. $t$, such that $u\left(  \text{\textperiodcentered}\right)
$ is the unique strong solution of the problem%
\begin{equation}
	\left\{
	\begin{array}
		[c]{c}%
		\dfrac{du}{dt}+Au=h(t),\text{ }t>0,\\
		u(0)=u_{0}\in L^{2}\left(  \Omega\right)  ,
	\end{array}
	\right.  \label{Sel}%
\end{equation}
which means that $u(\cdot)$ is absolutely continuous on any compact subset of
$(0,T)$, it is almost everywhere (a.e.) differentiable on $(0,T)$, and
$u(\cdot)$ satisfies the equation in \eqref{Sel} a.e. on $(0,T)$. Denote the
solution of problem \eqref{Sel} by $u\left(  \text{\textperiodcentered
}\right)  =I(u_{0})h(\cdot)$. It is known \cite{Barbu} that for any
$u_{i}(\cdot)=I(u_{0}^{i})h_{i}(\cdot),$ $i=1,2,$ the next inequality holds:
\begin{equation}
	\left\Vert u_{1}(t)-u_{2}(t)\right\Vert _{L^{2}}\leq\left\Vert u_{1}%
	(s)-u_{2}(s)\right\Vert _{L^{2}}+\int_{s}^{t}\left\Vert h_{1}(\tau)-h_{2}%
	(\tau)\right\Vert _{L^{2}}d\tau,\,\,t\geq s. \label{Int.Ineq2.}%
\end{equation}

If we fix $T>0$, it is known \cite{Tolstonogov} that for any $z(\cdot
)=I(z_{0})g(\cdot)$ and any $u_{0}\in L^{2}\left(  \Omega\right)  $ there
exists a solution $u(\cdot)=I(u_{0})h(\cdot)$ of problem \eqref{Inclusion3}
such that
\begin{equation}
	\left\Vert u(t)-z(t)\right\Vert _{L^{2}}\leq\xi(t),\,\forall t\in\lbrack0,T],
	\label{Int.Ineq3.}%
\end{equation}%
\begin{equation}
	\left\Vert h(t)-g(t)\right\Vert \leq\rho(t)+2C\xi(t),\,\,\text{a.e. on }(0,T),
	\label{Int.Ineq4.}%
\end{equation}
where
\[
\rho(t)=2dist\left(  g(t),F(z(t))\right)  ,
\]%
\[
\xi(t)=\left\Vert u_{0}-z_{0}\right\Vert _{L^{2}}\exp(2Ct)+\int_{0}^{t}%
\exp(2C(t-s))\rho(s)ds.
\]
Concatenating solutions we can easily obtain a solution satisfying these
inequalities for any $T>0$.

\begin{lemma}
	\label{Approx}Let $u\left(  \text{\textperiodcentered}\right)  =I(u_{0}
	)h\left(  \text{\textperiodcentered}\right)  $ be a solution to problem
	\eqref{Inclusion3}. Then for any sequence $u_{0}^{n}\rightarrow u_{0}$ in
	$L^{2}\left(  \Omega\right)  $ there exists a sequence of solutions
	$u_{n}\left(  \text{\textperiodcentered}\right)  =I(u_{0}^{n})h_{n}\left(
	\text{\textperiodcentered}\right)  $ of problem \eqref{Inclusion3} such that
	$u_{n}\rightarrow u$ in $C([0,T],L^{2}\left(  \Omega\right)  )$ for every
	$T>0.$
\end{lemma}

\begin{proof}
	Since {$h\left(  t\right)  \in F(u(t))$} for a.a. $t$, we have
	\[
	\rho(t)=2dist\left(  h(t),F(u(t))\right)  =0\text{ for a.a. }t,
	\]
	so in view of \eqref{Int.Ineq3.} for each $u_{0}^{n}$ there exist solutions
	$u_{n}\left(  \text{\textperiodcentered}\right)  =I(u_{0}^{n})h_{n}\left(
	\text{\textperiodcentered}\right)  $ of problem \eqref{Inclusion3} such that
	\[
	\left\Vert u(t)-u_{n}(t)\right\Vert _{L^{2}}\leq\left\Vert u_{0}-u_{0}%
	^{n}\right\Vert _{L^{2}}\exp(2Ct),\,\forall t\geq0.
	\]
	Then the result follows.
\end{proof}

\begin{corollary}
	\label{InclK5}Property (K5) is satisfied in $L^{2}\left(  \Omega\right)  $.
\end{corollary}

\begin{lemma}
	Let $u\left(  \text{\textperiodcentered}\right)  =I(u_{0})h\left(
	\text{\textperiodcentered}\right)  $ be a solution to problem
	\eqref{Inclusion3} with $u_{0}\in H_{0}^{1}\left(  \Omega\right)  $. Then for
	any sequence $u_{0}^{n}\rightarrow u_{0}$ in $H_{0}^{1}\left(  \Omega\right)
	$ there exists a sequence of solutions $u_{n}\left(  \text{\textperiodcentered
	}\right)  =I(u_{0}^{n})h_{n}\left(  \text{\textperiodcentered}\right)  $ of
	problem \eqref{Inclusion3} such that $u_{n}\rightarrow u$ in $C([0,T],H_{0}
	^{1}\left(  \Omega\right)  )$ for every $T>0.$
\end{lemma}

\begin{proof}
	From Lemma \ref{Approx} we obtain the sequence $u_{n}\left(
	\text{\textperiodcentered}\right)  =I(u_{0}^{n})h_{n}\left(
	\text{\textperiodcentered}\right)  $. Since $u_{0}^{n}\rightarrow u_{0}$ in
	$H_{0}^{1}\left(  \Omega\right)  $ we can prove in a standard way that
	$u_{n}\rightarrow u$ in $C([0,T],H_{0}^{1}\left(  \Omega\right)  ).$
\end{proof}

\begin{corollary}
	Property (K5) is satisfied in $H_{0}^{1}\left(  \Omega\right)  $.
\end{corollary}

\section{Application}

Let us consider the differential inclusion
\begin{equation}
	\left\{
	\begin{array}
		[c]{l}%
		\dfrac{\partial u}{\partial t}-\dfrac{\partial^{2}u}{\partial x^{2}}\in
		H_{0}(u)+\omega u,\ \text{on\ }(0,\infty)\times\Omega,\\
		u|_{\partial\Omega}=0,\\
		u(x,0)=u_{0}(x),
	\end{array}
	\right.  \label{Incl}%
\end{equation}
where $\Omega=(0,1)$, $0\leq\omega<\pi^{2}$, and
\[
H_{0}(u)=\left\{
\begin{array}
	[c]{ll}%
	-1, & \text{if }u<0,\\
	\left[  -1,1\right]  , & \text{if }u=0,\\
	1, & \text{if }u>0
\end{array}
\right.
\]
is the Heaviside function. Differential inclusions of the type appear when we
have a reaction-diffusion equation with a discontinuous nonlinearity and we
complete the image of the function at the points of discontinuity with a
vertical line. Equations of this type appear in models of physical interest
(see, for example, \cite{FeNo,NC,Terman1,Terman2}).

In this section, we will prove the existence of isolating blocks for the fixed
points (but $0$) of problem \eqref{Incl} by using the results of Section
\ref{Block}. Also, we will prove a uniqueness theorem for initial conditions
of certain type.

\subsection{Previous results}

We recall what is known about the dynamics of problem \eqref{Incl}.

Problem \eqref{Incl} can be written in a functional form. Indeed, we define
the following proper, convex, lower semicontinuous functions $\psi^{i}\colon
L^{2}(\Omega)\rightarrow(-\infty,+\infty]$:
\[
\psi^{1}\left(  u\right)  =\left\{
\begin{array}
	[c]{c}%
	\frac{1}{2}\int_{\Omega}\left\vert \nabla u\right\vert ^{2}dx,\text{ if }u\in
	H_{0}^{1}\left(  \Omega\right)  ,\\
	+\infty\text{, otherwise,}%
\end{array}
\right.
\]%
\[
\psi^{2}\left(  u\right)  =\left\{
\begin{array}
	[c]{c}%
	\int_{\Omega}\left(  \omega\frac{u^{2}}{2}+\left\vert u\right\vert \right)
	dx,\text{ if }\left\vert u\left(  \text{\textperiodcentered}\right)
	\right\vert \in L^{1}\left(  \Omega\right)  ,\\
	+\infty\text{, otherwise.}%
\end{array}
\right.
\]
It is known (see e.g. \cite{Barbu}) that the subdifferentials $\partial
\psi^{1}$ and $\partial\psi^{2}$ of these functions are given by%
\[
\partial\psi^{1}(u)=\left\{  y\in L^{2}(\Omega):y(x)=-\frac{\partial^{2}
	u}{\partial x^{2}}(x)\text{, a.e. on }\Omega\right\}  ,
\]
\[
\partial\psi^{2}(u)=\left\{  y\in L^{2}\left(  \Omega\right)  :\text{
}y\left(  x\right)  \in H_{0}\left(  u\left(  x\right)  \right)  +\omega
u\left(  x\right)  \text{, a.e. on }\Omega\right\}  .
\]
Hence, problem \eqref{Incl} can be rewritten in the abstract form
\begin{equation}
	\left\{
	\begin{array}
		[c]{l}%
		\dfrac{\partial u}{\partial t}+\partial\psi^{1}(u)-\partial\psi^{2}(u)\ni0,\\
		u(0)=u_{0}.
	\end{array}
	\right.  \label{Abstract}%
\end{equation}
We observe that $\left\vert u\right\vert =\int_{0}^{u}H_{0}\left(  s\right)
ds $ and $D\left(  \partial\psi^{1}\right)  =H^{2}\left(  \Omega\right)  \cap
H_{0}^{1}\left(  \Omega\right)  ,\ D\left(  \partial\psi^{2}\right)
=L^{2}\left(  \Omega\right)  $.

\begin{definition}
	For $u_{0}\in L^{2}\left(  \Omega\right)  $ and $T>0$ the function $u\in
	C(\left[  0,T\right]  ,L^{2}(\Omega))$ is called a strong solution of problem
	\eqref{Incl} on $[0,T]$ if:
	
	\begin{enumerate}
		\item $u(0)=u_{0}$;
		
		\item $u(\cdot)$ is absolutely continuous on $(0,T)$ and $u\left(  t\right)
		\in D\left(  \partial\psi^{1}\right)  $ for a.a. $t\in\left(  0,T\right)  $;
		
		\item There exist a function $g\in L^{2}\left(  0,T;L^{2}\left(
		\Omega\right)  \right)  $ such that $g(t)\in\partial\psi^{2}(u(t))$, a.e. on
		$(0,T)$, and
		\begin{equation}
			\frac{du(t)}{dt}-\frac{\partial^{2}u\left(  t\right)  }{\partial x^{2}
			}-g(t)=0,\text{ for a.a. }t\in(0,T), \label{exist1}%
		\end{equation}
		where the equality is understood in the sense of the space $L^{2}(\Omega).$
	\end{enumerate}
\end{definition}

\begin{remark}
	Alternatively, equality \eqref{exist1} can be written as
	\begin{equation}
		\frac{du(t)}{dt}-\frac{\partial^{2}u\left(  t\right)  }{\partial x^{2}
		}-h(t)=\omega u\left(  t\right)  ,\text{ for a.a. }t\in(0,T), \label{exist2}%
	\end{equation}
	where $h\in L^{2}\left(  0,T;L^{2}\left(  \Omega\right)  \right)  $ and
	$h(t,x)\in H_{0}(u(t,x))$, for a.e. $t>0$, $x\in\Omega$.
\end{remark}

From \cite[Theorem 4, Lemmas 1 and 2]{Valero01} we know the following facts.
For each $u_{0}\in L^{2}\left(  \Omega\right)  $ and $T>0$ there exists at
least one strong solution $u\left(  \cdot\right)  $ of \eqref{Incl} and each
solution can be extended to the whole semiline $[0,\infty)$, so that they are
global. Moreover, any solution $u\left(  \text{\textperiodcentered}\right)  $
belongs to the space $C\left(  \left(  0,+\infty\right)  ,H_{0}^{1}\left(
\Omega\right)  \right)  $ and, if $u_{0}\in H_{0}^{1}\left(  \Omega\right)  $,
then $u\in C\left(  [0,+\infty),H_{0}^{1}\left(  \Omega\right)  \right)  $.

Let $\mathcal{D}\left(  u_{0}\right)  $ be the set of all strong solutions
defined on $\left[  0,+\infty\right)  $ for the initial condition $u_{0}$ and
let $\mathcal{R}=\cup_{u_{0}\in L^{2}\left(  \Omega\right)  }\mathcal{D}%
\left(  u_{0}\right)  $. Let $G:\mathbb{R}^{+}\times L^{2}\left(
\Omega\right)  \rightarrow P\left(  L^{2}\left(  \Omega\right)  \right)  $ be
the map%
\[
G\left(  t,u_{0}\right)  =\{u\left(  t\right)  :u\in\mathcal{D}\left(
u_{0}\right)  \},
\]
which is a strict multivalued semiflow. Moreover, properties $\left(
K1\right)  -\left(  K3\right)  $ are satisfied for $\mathcal{R}$. Also,
$\left(  K4\right)  $ is shown to be true in \cite[Lemma 31]{HenVal}.

Concerning the asymptotic behavior of solutions in the long term, $G$
possesses a global compact invariant attractor $\mathcal{A}$ \cite[Theorem
4]{Valero01}, which is characterized by the union of all bounded complete
trajectories. In addition, $\mathcal{A}$ is compact in $W^{2-\delta,p}\left(
\Omega\right)  $ for all $\delta>0$,\ $p\geq1$ and
\[
dist_{W^{2-\delta,p}}\left(  G(t,B),\mathcal{A}\right)  \rightarrow0,\text{ as
}t\rightarrow+\infty,
\]
for any bounded set $B$ \cite{ARV06}. It follows then that $\mathcal{A}$ is
compact in $C^{1}\left(  [0,1]\right)  $ and\\ $dist_{C^{1}}\left(
G(t,B),\mathcal{A}\right)  \rightarrow0$ as $t\rightarrow+\infty$. Also, it is
proved in \cite{Valero05} that $\mathcal{A}$ is a connected set.

The structure of the attractor was studied in detail in \cite{ARV06}. We
summarize the main results. Problem \eqref{Incl} has an infinite (but
countable) number of fixed points:\ $v_{0}\equiv0$,\ $v_{1}^{+},\ v_{1}%
^{-},\ v_{2}^{+},\ v_{2}^{-},...$ ($v_{k}^{+}\left(  x\right)  =-v_{k}%
^{-}\left(  x\right)  $ for all $k$). which satisfy the following properties:

\begin{enumerate}
	\item $v_{k}^{\pm}$ possess exactly $k-1$ zeros in $\left(  0,1\right)  ;$
	
	\item $v_{1}^{+},v_{1}^{-}$ are asymptotically stable (so for $u_{0}
	=v_{1}^{\pm}$ the solution is unique);
	
	\item $0,\ v_{k}^{\pm},\ k\geq2$, are unstable;
	
	\item $v_{k}^{\pm}\rightarrow0$ as $k\rightarrow\infty.$
\end{enumerate}

We define the continuous function $E:H_{0}^{1}\left(  0,1\right)
\rightarrow\mathbb{R}$ by
\begin{equation}
	E\left(  u\right)  =\frac{1}{2}\int_{0}^{1}\left\vert \frac{\partial
		u}{\partial x}\right\vert ^{2}dx-\int_{0}^{1}\left(  \left\vert u\right\vert
	+\frac{\omega}{2}u^{2}\right)  dx=\psi^{1}\left(  u\right)  -\psi^{2}\left(
	u\right)  . \label{LyapunovFunction}%
\end{equation}
It is shown in in \cite{ARV06} that $E$ is a Lyapunov function and then that
for any $u\in\mathcal{D}\left(  u_{0}\right)  $, $u_{0}\in L^{2}\left(
\Omega\right)  $, there is a fixed point $z$ such that $u\left(  t\right)
\rightarrow z$ as $t\rightarrow+\infty$. We note that by the regularity of the
solutions, $E\left(  u\left(  t\right)  \right)  :\left(  0,+\infty\right)
\rightarrow\mathbb{R}$ is a continuous function. We note also that if
$u_{0}\in H_{0}^{1}(\Omega)$, then $E\left(  u\left(  t\right)  \right)  $ is
continuous on $\left[  0,+\infty\right)  $. Also, if $\phi$ is a bounded
complete trajectory, then there is a fixed point $z$ such that $\phi\left(
t\right)  \rightarrow z$ as $t\rightarrow-\infty$. Therefore, the global
attractor is characterized by the set of stationary points and their
heteroclinic connections. In \cite{ARV06} some of these connections have been
established, although the question of determining the full set of connections
is still open. The fixed points are ordered by the Lyapunov function $E$ in
the following way:%
\[
E\left(  v_{1}\right)  =E\left(  v_{1}^{-}\right)  <E\left(  v_{2}\right)
=E\left(  v_{2}^{-}\right)  <...<E\left(  v_{k}\right)  =E\left(  v_{k}%
^{-}\right)  <...<E\left(  0\right)  =0.
\]
In particular, this implies that\ heteroclinic connections from $v_{k}^{\pm}$
to $v_{j}^{\pm}$ with $k\leq j$ are forbidden. Finally, we observe that the
fixed point $0$ is special, because for any other fixed point $z=v_{k}^{+}$
(or $v_{k}^{-}$) there exists a solution $u\left(  \text{\textperiodcentered
}\right)  $ starting at $0$ such that $u\left(  t\right)  \rightarrow z$ as
$t\rightarrow+\infty$. The conclusion is two-fold: on the one hand, for the
initial condition $u_{0}=0$ there exists an infinite number of solutions; on
the other hand, for any $z=v_{k}^{+}$ (or $v_{k}^{-}$) there exists an
heteroclinic connection from $0$ to $z.$

\subsection{Isolating block}

In order to understand the dynamics inside of the global attractor it is
important to know what happens in a neighborhood of each fixed point.
Reasoning as in \cite[p.32]{HenVal} we can establish that each fixed point
$v_{k}^{+}$ (or $v_{k}^{-}$), $k\geq1$, is an isolated weakly invariant set.
The point $0$ is not isolated as $v_{k}^{\pm}\rightarrow0$ as $k\rightarrow
\infty.$ Applying the results of the previous section we will obtain the
existence of an isolating block for each $v_{k}^{+}$ ($v_{k}^{-}$), $k\geq1$.

It is not possible to apply directly Theorem \ref{theo:Isol.Block.(K1)-(K5)},
because the solutions of \eqref{Incl} do not satisfy condition $\left(
K5\right)  $, as the following lemma shows.

\begin{lemma}
	There exists a sequence $\{u_{0}^{n}\}$ and a solution $u\left(
	\text{\textperiodcentered}\right)  \in\mathcal{D}(0)$ such that $u_{0}
	^{n}\rightarrow0$ and there is no subsequence of solutions $\{u^{n_{k}}\left(
	\text{\textperiodcentered}\right)  \}$ with $u^{n_{k}}(0)=u_{0}^{n_{k}}$ such
	that $u^{n_{k}}\rightarrow u$ (in the sense of $\left(  K5\right)  $).
\end{lemma}

\begin{proof}
	Let $u_{0}^{n}\in V^{2r}$, $\frac{3}{4}<r<1$, where as usual $V^{2r}=D\left(
	A^{r}\right)  $ and $A,$ $D\left(  A\right)  =H^{2}\left(  \Omega\right)  \cap
	H_{0}^{1}\left(  \Omega\right)  $, is the operator $-\dfrac{d^{2}}{dx^{2}}$
	with Dirichlet boundary conditions. We choose $u_{0}^{n}$ such that $\dfrac
	{d}{dx}u_{0}^{n}\left(  0\right)  >0,\ \dfrac{d}{dx}u_{0}^{n}\left(  0\right)
	<0,\ u_{0}^{n}\left(  x\right)  >0$ for $x\in\left(  0,1\right)  $ (we observe
	that $V^{2r}\subset C^{1}\left(  \overline{\Omega}\right)  $) and $u_{0}
	^{n}\rightarrow0$. Then by \cite[Lemma 13]{Valero2021} there exists a unique
	solution $u^{n}\left(  \text{\textperiodcentered}\right)  \in\mathcal{D}
	(u_{0}^{n})$ which satisfies $u^{n}\left(  t,x\right)  >0$ for any
	$x\in\left(  0,1\right)  $ and $t\geq0$. Also, it converges to $v_{1}^{+}$ as
	$t\rightarrow+\infty$. We know from \cite[Theorem 6.7]{ARV06} that there
	exists a solution $v\left(  \text{\textperiodcentered}\right)  $ such that
	$v\left(  0\right)  =0$ and $v\left(  t\right)  \rightarrow v_{1}^{-}$ in
	$C^{1}(\overline{\Omega})$ as $t\rightarrow+\infty$. It is clear that no
	subsequence of $u^{n}\left(  \text{\textperiodcentered}\right)  $ can converge
	to $v\left(  \text{\textperiodcentered}\right)  $, because $u^{n}\left(
	t\right)  $ is positive for any $t\geq0$ but $v\left(  t\right)  $ take
	negative values for $t$ large enough.
\end{proof}

\bigskip

Hence, we will apply Theorem \ref{theo:exist_block_subsemiflow}. For this aim,
we need to define a semiflow $\widetilde{G}$ containing $G$ that satisfies
$\left(  K1\right)  -\left(  K5\right)  $.

For any $\varepsilon>0$ let us define the multivalued function $g_{\varepsilon
}$ given by%
\[
g_{\varepsilon}\left(  u\right)  =\left\{
\begin{array}
	[c]{c}%
	-1\text{ if }u\leq-\varepsilon,\\
	\lbrack-1,\frac{2}{\varepsilon}u+1]\text{ if }-\varepsilon\leq u\leq0,\\
	\lbrack\frac{2}{\varepsilon}u-1,1]\text{ if }0\leq u\leq\varepsilon,\\
	1\text{ if }u\geq\varepsilon.
\end{array}
\right.
\]
It is easy to see that the map $f_{\varepsilon}\left(  u\right)
=g_{\varepsilon}\left(  u\right)  +\omega u$ satisfies conditions
$(f1)$-$(f2)$ for problem \eqref{Inclusion2}. Then problem \eqref{Inclusion2}
with $f=f_{\varepsilon}$ and $q=0$ generates for each $\varepsilon>0$ a strict
multivalued semiflow $G_{\varepsilon}$ which contains the semiflow $G$ for
problem \eqref{Incl} (as every solution to problem \eqref{Incl} is obviously a
solution to problem \eqref{Inclusion2}).

We denote by $\mathcal{D}_{\varepsilon}\left(  u_{0}\right)  $ the set of all
strong solutions defined on $\left[  0,+\infty\right)  $ for the initial
condition $u_{0}$. Let $\mathcal{R}_{\varepsilon}=\cup_{u_{0}\in L^{2}\left(
	\Omega\right)  }\mathcal{D}_{\varepsilon}\left(  u_{0}\right)  $. It follows
from the proof of Lemma 6 in \cite{MelnikValero98} that $\left(  K1\right)
-\left(  K3\right)  $ hold true. In view of Corollary \ref{InclK5}, $\left(
K5\right)  $ is satisfied. We prove that $\left(  K4\right)  $ holds as well.

\begin{lemma}
	\label{PropertyK4}$\left(  K4\right)  $ is satisfied.
\end{lemma}

\begin{proof}
	Let $u_{0}^{n}\rightarrow u_{0}$. In view of \eqref{Int.Ineq3.}, for any
	$u^{n}\left(  \text{\textperiodcentered}\right)  \in\mathcal{D}_{\varepsilon
	}\left(  u_{0}^{n}\right)  $ there exists $u_{n}\left(
	\text{\textperiodcentered}\right)  \in\mathcal{D}_{\varepsilon}\left(
	u_{0}\right)  $ such that
	\begin{equation}
		\left\Vert u^{n}(t)-u_{n}(t)\right\Vert _{L^{2}}\leq\left\Vert u_{0}^{n}
		-u_{0}\right\Vert _{L^{2}}\exp(2Ct),\,\forall t\geq0. \label{Ineq}%
	\end{equation}
	Fix $T>0$. Let $\pi_{T}\mathcal{D}_{\varepsilon}\left(  u_{0}\right)  $ be the
	restriction of $\mathcal{D}_{\varepsilon}\left(  u_{0}\right)  $ onto
	$C([0,T],L^{2}(\Omega))$. Since the set $\pi_{T}\mathcal{D}_{\varepsilon
	}\left(  u_{0}\right)  $ is compact in $C([0,T],L^{2}(\Omega))$ \cite[p.100]%
	{MelnikValero98}, passing to a subsequence we have that $u_{n}\rightarrow
	u\in\mathcal{D}_{\varepsilon}\left(  u_{0}\right)  $ in $C([0,T],L^{2}
	(\Omega))$. Thus, by \eqref{Ineq} we obtain that $u^{n}\rightarrow u$ in
	$C([0,T],L^{2}(\Omega))$. By a diagonal argument we deduce that for some
	subsequence this is true for any $T>0$, proving property $\left(  K4\right)  .
	$
\end{proof}

\bigskip

From \cite{MelnikValero98} we know that $G_{\varepsilon}$ has a global compact
invariant attractor $\mathcal{A}_{\varepsilon}$. It is clear that%
\[
\mathcal{A}\subset\mathcal{A}_{\varepsilon_{1}}\subset\mathcal{A}%
_{\varepsilon_{2}}\text{ for all }0<\varepsilon_{1}<\varepsilon_{2},
\]
where $\mathcal{A}$ is the attractor for problem \eqref{Incl}. Also, as
$\left(  K1\right)  -\left(  K4\right)  $ hold, $\mathcal{A}_{\varepsilon}$ is
characterized by the union of all bounded global trajectories \cite{KKV14}:%
\[
\mathcal{A}_{\varepsilon}=\{\phi\left(  0\right)  :\phi\text{ is a bounded
	complete trajectory of }\mathcal{R}_{\varepsilon}\}.
\]

\begin{lemma}
	\label{ConvSol}If $\varepsilon_{n}\rightarrow0^{+}$, $u_{\varepsilon_{n}}
	\in\mathcal{D}_{\varepsilon_{n}}\left(  u_{0}^{n}\right)  $ and $u_{0}
	^{n}\rightarrow u_{0}$, then up to a subsequence $u_{\varepsilon_{n}
	}\rightarrow u\in\mathcal{D}\left(  u_{0}\right)  $ uniformly on bounded sets
	of $[0,+\infty).$
\end{lemma}

\begin{proof}
	We fix $\varepsilon_{0}>0$ such that $\varepsilon_{n}<\varepsilon_{0}$. Since
	$u_{\varepsilon_{n}}\in\mathcal{D}_{\varepsilon_{0}}\left(  u_{0}^{n}\right)
	$ for all $n$, by Lemma \ref{PropertyK4} we obtain that up to a subsequence
	$u_{\varepsilon_{n}}\rightarrow u\in\mathcal{D}_{\varepsilon_{0}}\left(
	u_{0}\right)  $ uniformly on bounded sets of $[0,+\infty).$ Hence, $u\left(
	\text{\textperiodcentered}\right)  $ is a strong solution to problem
	\eqref{Sel} with $h\in L_{loc}^{2}(0,+\infty;L^{2}\left(  \Omega\right)  )$,
	$h\left(  t\right)  \in F_{\varepsilon_{0}}(u(t))$ for a.a. $t$, where
	$F_{\varepsilon_{0}}$ is the map \eqref{F} for $f_{\varepsilon_{0}}$.
	
	In order to prove that $u\in\mathcal{D}\left(  u_{0}\right)  $, it remains to
	show that $h\left(  t,x\right)  \in H_{0}\left(  u\left(  t,x\right)  \right)
	+\omega u\left(  t,x\right)  $ for a.a. $\left(  t,x\right)  $.
	
	The selections $h_{n}\left(  \text{\textperiodcentered}\right)  $
	corresponding to $u_{\varepsilon_{n}}\left(  \text{\textperiodcentered
	}\right)  $ in equality \eqref{Sel} are bounded by a constant $C_{T}$ in each
	interval $[0,T]$:
	\[
	\left\Vert h_{n}\left(  t\right)  \right\Vert _{L^{2}}\leq C_{T}\text{ for
		a.a. }t\in(0,T).
	\]
	In particular, this means that $h_{n}$ are integrably bounded in each interval
	and that up to a subsequence $h_{n}\rightarrow\widetilde{h}$ weakly in
	$L^{2}(0,T;L^{2}\left(  \Omega\right)  )$ for any $T>0$. We need to check that
	$\widetilde{h}=h$. Let $v_{n}\left(  \text{\textperiodcentered}\right)
	=I\left(  u_{0}\right)  h_{n}\left(  \text{\textperiodcentered}\right)  $.
	Then by inequality \eqref{Int.Ineq2.} we have that $v_{n}\rightarrow u$ in
	$C([0,T],L^{2}\left(  \Omega\right)  )$ for any $T>0$. By Lemma 1.3 in
	\cite{Tolstonogov} we deduce that $u\left(  \text{\textperiodcentered}\right)
	=I\left(  u_{0}\right)  \widetilde{h}\left(  \text{\textperiodcentered
	}\right)  $, which is possible if and only if $\widetilde{h}=h$.
	
	Denote $g\left(  t\right)  =h\left(  t\right)  -\omega u\left(  t\right)  $
	and $g_{n}\left(  t\right)  =h_{n}\left(  t\right)  -\omega u_{n}\left(
	t\right)  $. We need to prove that $g\left(  t,x\right)  \in H_{0}\left(
	u\left(  t,x\right)  \right)  $ for a.a. $\left(  t,x\right)  $. For a.a.
	$\left(  t,x\right)  $ there is $N\left(  t,x\right)  $ such that
	$g_{n}(t,x)\in H_{0}\left(  u\left(  t,x\right)  \right)  $ if $n\geq N\left(
	t,x\right)  $. Indeed, since $u_{n}\left(  t,x\right)  \rightarrow u\left(
	t,x\right)  $ for a.a. $\left(  t,x\right)  $, we define $B$ as a set which
	complementary $B^{c}$ has measure $0$ and such that $u_{n}\left(  t,x\right)
	\rightarrow u\left(  t,x\right)  $ for $u\left(  t,x\right)  \in B$. If
	$u\left(  t,x\right)  \in B$ and $u\left(  t,x\right)  >0$ ($<0$), then there
	is $N\left(  t,x\right)  $ such that $u_{n}\left(  t,x\right)  >0$ ($<0$) for
	$n\geq N\left(  t,x\right)  $. Hence, $g_{n}\left(  t,x\right)  \in
	H_{0}\left(  u_{n}(t,x)\right)  =H_{0}\left(  u\left(  t,x\right)  \right)
	=1$ ($-1$). If $u\left(  t,x\right)  \in B$ and $u\left(  t,x\right)  =0$,
	then $g_{n}\left(  t,x\right)  \in\lbrack-1,1]=H_{0}\left(  u\left(
	t,x\right)  \right)  $ for all $n$. By \cite[Proposition 1.1]{Tolstonogov} for
	a.a. $t$ there is a sequence of convex combinations
	\[
	y_{n}\left(  t\right)  =\sum_{j=1}^{N_{n}}\lambda_{j}g_{k_{j}}\left(
	t\right)  ,\ \sum_{j=1}^{N_{n}}\lambda_{j}=1,\ k_{j}\geq n,
	\]
	such that $y_{n}\left(  t\right)  \rightarrow g\left(  t\right)  $ in
	$L^{2}\left(  \Omega\right)  $. Then, as $H_{0}\left(  u\left(  t,x\right)
	\right)  $ is closed and convex, $g\left(  t,x\right)  \in H_{0}\left(
	u\left(  t,x\right)  \right)  $ for a.a. $\left(  t.x\right)  .$
\end{proof}

\begin{corollary}
	\label{ConvergGlobalTray}If $\{\phi_{\varepsilon_{n}}\}$ is a sequence of
	bounded global trajectories of $\mathcal{R}_{\varepsilon_{n}}$ and
	$\varepsilon_{n}\rightarrow0^{+}$, then there exists a subsequence
	$\{\phi_{\varepsilon_{n_{k}}}\}$ and a bounded complete trajectory $\phi$ of
	$\mathcal{R}$ such that
	\begin{equation}
		\phi_{\varepsilon_{n_{k}}}\rightarrow\phi\text{ in }C([-T,T],L^{2}\left(
		\Omega\right)  )\text{ for all }T>0. \label{ConvGlobalSol}%
	\end{equation}
	
\end{corollary}

\begin{proof}
	Applying Lemma \ref{ConvSol} and a diagonal argument we obtain a complete
	trajectory of $\mathcal{R}$ and a subsequence such that \eqref{ConvGlobalSol}%
	\ holds. Since for $\varepsilon_{0}>0$ the complete trajectory $\phi$ belongs
	to $\mathcal{A}_{\varepsilon_{0}}$ and $\mathcal{A}_{\varepsilon_{0}}$ is
	bounded, we obtain that $\phi$ is a bounded complete trajectory of
	$\mathcal{R}$.
\end{proof}

\bigskip

We denote by $O_{\delta}(v_{0})=\{v\in X:\left\Vert v-v_{0}\right\Vert
_{L^{2}}<\delta\}$ a $\delta$-neighborhood of the point $v_{0}\in L^{2}%
(\Omega).$

We choose $\delta>0$ such that $v_{k}^{+}$ is the maximal weakly invariant set
in $O_{\delta}\left(  v_{k}^{+}\right)  $, so that $\overline{O}_{\delta
}\left(  v_{k}^{+}\right)  $ is an isolating closed neighborhood of\ the
stationary point $v_{k}^{+}$, $k\geq1$ (for $v_{k}^{-}$ the proof is the
same). For the semiflow $G_{\varepsilon}$ we define a weakly invariant set
associated to $v_{k}^{+}$ in the following way:%
\[\begin{aligned}
K_{\varepsilon}=\{\phi\left(  0\right)  :\phi\left(  \text{\textperiodcentered
}\right)  \text{ is a bounded complete trajectory of }\mathcal{R}%
_{\varepsilon}\text{ such that }\phi\left(  t\right)  \in\overline{O}_{\delta
}\left(  v_{k}^{+}\right) \ \\  \text{ for all }t\in\mathbb{R}\}.
\end{aligned}
\]

\begin{lemma}
	\label{compact}The set $K_{\varepsilon}$ is compact.
\end{lemma}

\begin{proof}
	Since $K_{\varepsilon}\subset\mathcal{A}_{\varepsilon}$, it is clearly
	relatively compact. Thus, we need to prove just that it is closed. Let
	$y_{n}\rightarrow y$, where $y_{n}\in K_{\varepsilon}$. Then $y_{n}=\phi
	_{n}\left(  0\right)  $ for some bounded complete trajectory $\phi_{n}$.
	By Lemma \ref{PropertyK4} and arguing as in Corollary
		\ref{ConvergGlobalTray} we obtain that up to a subsequence $\phi
	_{n}\rightarrow\phi$ in $C([-T,T],L^{2}\left(  \Omega\right)  )$ for all
	$T>0$, where $\phi$ is a bounded complete trajectory of $\mathcal{R}%
	_{\varepsilon}$. Obviously, $\phi\left(  t\right)  \in\overline{O}_{\delta
	}\left(  v_{k}^{+}\right)  $ for all $t\in\mathbb{R}$. Hence, $y\in
	K_{\varepsilon}.$
\end{proof}

\begin{lemma}
	\label{deltamedios}There is $\varepsilon_{0}>0$ such that $K_{\varepsilon
	}\subset O_{\delta/2}\left(  v_{k}^{+}\right)  $ for all $\varepsilon
	\leq\varepsilon_{0}.$
\end{lemma}

\begin{proof}
	By contradiction, if this is not true, there is a sequence of bounded global
	trajectories $\phi_{\varepsilon_{n}}$ of $\mathcal{R}_{\varepsilon_{n}}$,
	where $\varepsilon_{n}\rightarrow0^{+}$, and times $t_{n}$ such that
	$\phi_{\varepsilon_{n}}\left(  \mathbb{R}\right)  \subset\overline{O}_{\delta
	}\left(  v_{k}^{+}\right)  $ and $\phi\left(  t_{n}\right)  \not \in
	O_{\delta/2}\left(  v_{k}^{+}\right)  $. Making use of Corollary
	\ref{ConvergGlobalTray} and the fact that $v_{k}^{+}$ is the unique bounded
	complete trajectory in $O_{\delta}\left(  v_{k}^{+}\right)  $ for
	$\mathcal{R}$, we conclude that $\phi_{\varepsilon_{n}}\rightarrow v_{k}^{+}$
	in $C([-T,T],L^{2}\left(  \Omega\right)  )$ for all $T>0$. Suppose
		that a subsequence tends to $+\infty$. Then we define the sequence
	$v_{\varepsilon_{n}}\left(  \text{\textperiodcentered}\right)  =\phi
	_{\varepsilon_{n}}\left(  \text{\textperiodcentered}+t_{n}\right)  $, which
	again by Corollary \ref{ConvergGlobalTray} converges in $C([-T,T],L^{2}\left(
	\Omega\right)  )$ to a bounded global trajectory $\phi$ of $\mathcal{R}$ such
	that $\phi\left(  \mathbb{R}\right)  \subset\overline{O}_{\delta}\left(
	v_{k}^{+}\right)  $, so that $\phi=v_{k}^{+}$. But then $v_{\varepsilon_{n}%
	}\left(  0\right)  =\phi_{\varepsilon_{n}}\left(  t_{n}\right)  \rightarrow
	v_{k}^{+}$, which is a contradiction. But if $\{t_{n}\}$ is bounded by a
	similar argument we obtain a contradiction.
\end{proof}

\begin{lemma}
	\label{maxInv}There is $\varepsilon_{0}>0$ such that $K_{\varepsilon}$ is the
	maximal weakly invariant set in $\overline{O}_{\delta}\left(  v_{k}%
	^{+}\right)  $ for any $\varepsilon\leq\varepsilon_{0}.$ Hence,$\ \overline
	{O}_{\delta}\left(  v_{k}^{+}\right)  $ is an isolating neighborhood for
	$K_{\varepsilon}$.
\end{lemma}

\begin{proof}
	In view of Lemmas \ref{compact} and \ref{deltamedios}, $K_{\varepsilon}$ is
	closed and $\overline{O}_{\delta}\left(  v_{k}^{+}\right)  $ is a neighborhood
	of $K_{\varepsilon}$ such that $K_{\varepsilon}\subset int(\overline
	{O}_{\delta}\left(  v_{k}^{+}\right)  )$. It is obvious that $K_{\varepsilon}$
	is the maximal weakly invariant set in $\overline{O}_{\delta}\left(  v_{k}%
	^{+}\right)  .$
\end{proof}

\begin{remark}
	\label{StrAdm}Since the semiflows $G$, $G_{\varepsilon}$ possess a compact
	global attractor, it is clear that any neighborhood (in particular
	$\overline{O}_{\delta}\left(  v_{k}^{+}\right)  $) is admissible.
\end{remark}

\bigskip

We are now ready to prove the existence of an isolating block.

\begin{theorem}
	The stationary points $v_{k}^{\pm}$, $k\geq1$, possess an isolating block.
\end{theorem}

\begin{proof}
	It is a consequence of Lemma \ref{maxInv}, Remark \ref{StrAdm} and Theorem
	\ref{theo:exist_block_subsemiflow}.
\end{proof}

\subsection{Uniqueness of solutions}

In this section we will prove a general result on uniqueness of solutions
which allow us to obtain that in a suitable neighborhood of the fixed points
$v_{k}^{\pm}$, $k\geq1$, the solutions are unique while they remain inside it.
In particular, the solutions starting at the fixed points $v_{k}^{\pm}$ are unique.

The function $v\in H_{0}^{1}\left(  \Omega\right)  $ is non-degenerate if
there is $C>0$ and $\alpha_{0}>0$ such that
\begin{equation}
	\mu\{x\in\left(  0,1\right)  :\left\vert v\left(  x\right)  \right\vert
	\leq\alpha\}\leq C\alpha\text{ for all }\alpha\in\left(  0,\alpha_{0}\right)
	, \label{Nondegenerate}%
\end{equation}
where $\mu$ stands for the Lebesgue measure in $\mathbb{R}$. A strong solution
$u:[0,T]\rightarrow H_{0}^{1}\left(  \Omega\right)  $ is said to be
non-degenerate if there are $C>0$ and $\alpha_{0}>0$ (independent on $t$) such
that
\begin{equation}
	\mu\{x\in\left(  0,1\right)  :\left\vert u\left(  t,x\right)  \right\vert
	\leq\alpha\}\leq C\alpha\text{ for all }\alpha\in\left(  0,\alpha_{0}\right)
	\text{ and }t\in\lbrack0,T]. \label{Nondegenerate2}%
\end{equation}

For $z\in\mathbb{R}$ denote $z^{+}=\max\{0,z\}.$

\begin{lemma}
	\label{IneqH0}Let $u_{1},u_{2}\in L^{\infty}(0,1)$ and let either $u_{1}$ or
	$u_{2}$ be non-degenerate. Then for any $z_{1},z_{2}\in L^{\infty}(0,1)$
	satisfying $z_{i}(x)\in H_{0}(u_{i}(x))$, for a.a. $x\in(0,1)$, $i=1,2$, we
	have 
		\[
		\int_{0}^{1}(z_{1}(x)-z_{2}(x))(u_{1}(x)-u_{2}(x))^{+}dx\leq2D\left\Vert
		(u_{1}-u_{2})^{+}\right\Vert _{L^{\infty}}^{2},
		\]
		where $D=\max\{C,1/\alpha_{0}\}$ and $C,\alpha_{0}$ are the constants in
		\eqref{Nondegenerate} for $u_{1}$.
\end{lemma}

\begin{proof}
	If $\left\Vert u_{1}-u_{2}\right\Vert _{L^{\infty}}\geq\alpha_{0}$, then%
	\[
	\int_{0}^{1}(z_{1}(x)-z_{2}(x))(u_{1}(x)-u_{2}(x))^{+}dx\leq2\left\Vert
	(u_{1}-u_{2})^{+}\right\Vert _{L^{\infty}}\leq\frac{2}{\alpha_{0}}\left\Vert
	(u_{1}-u_{2})^{+}\right\Vert _{L^{\infty}}^{2},
	\]
	so let $\left\Vert u_{1}-u_{2}\right\Vert _{L^{\infty}}<\alpha_{0}$. We set%
	\begin{align*}
		A^{i}  &  =\{x\in(0,1):u_{i}(x)=0\},\\
		\Omega_{+}^{i}  &  =\{x\in(0,1):u_{i}(x)>0\},\\
		\Omega_{-}^{i}  &  =\{x\in(0,1):u_{i}(x)<0\},\\
		I  &  =\{x\in(0,1):u_{1}(x)>u_{2}(x)\}.
	\end{align*}
	Hence, $z_{1}(x)=z_{2}(x)$ for $x\in I_{2}=((\Omega_{+}^{1}\cap\Omega_{+}%
	^{2})\cup(\Omega_{-}^{1}\cap\Omega_{-}^{2}))\cap I$. Putting $I_{1}%
	=I\backslash I_{2}$, we have%
	\begin{align*}
		\int_{0}^{1}(z_{1}(x)-z_{2}(x))(u_{1}(x)-u_{2}(x))^{+}dx  &  =\int_{I_{1}%
		}(z_{1}(x)-z_{2}(x))(u_{1}(x)-u_{2}(x))^{+}dx\\
		&  \leq2\left\Vert (u_{1}-u_{2})^{+}\right\Vert _{L^{\infty}}\mu(I_{1}).
	\end{align*}

	We observe that $I_{1}=(A^{1}\cup A^{2}\cup(\Omega_{+}^{1}\cap\Omega_{-}%
	^{2}))\cap I$. Since
	\[
	0\leq u_{1}(x)\leq u_{2}(x)+\left\Vert u_{1}-u_{2}\right\Vert _{L^{\infty}%
	}\leq\left\Vert u_{1}-u_{2}\right\Vert _{L^{\infty}}\text{ for }x\in I_{1},
	\]
	if $u_{1}$ is non-degenerate, we obtain%
	\[
	\mu(I_{1})\leq\mu\{x:\left\vert u_{1}(x)\right\vert \leq\left\Vert u_{1}%
	-u_{2}\right\Vert _{L^{\infty}}\}\leq C\left\Vert u_{1}-u_{2}\right\Vert
	_{L^{\infty}}.
	\]
	In the same way, if $u_{2}$ is non-degenerate, then%
	\[
	\mu(I_{1})\leq\mu\{x:\left\vert u_{2}(x)\right\vert \leq\left\Vert u_{1}%
	-u_{2}\right\Vert _{L^{\infty}}\}\leq C\left\Vert u_{1}-u_{2}\right\Vert
	_{L^{\infty}}.
	\]

	Hence,%
	\[
	\int_{0}^{1}(z_{1}(x)-z_{2}(x))(u_{1}(x)-u_{2}(x))^{+}dx\leq2C\left\Vert
	u_{1}-u_{2}\right\Vert _{L^{\infty}}^{2}.
	\]

	Putting $D=\max\{C,1/\alpha_{0}\}$ the result follows.
\end{proof}

\begin{remark}
	\label{IneqH0Remark}Lemma \ref{IneqH0} is true if we change $\left(
	0,1\right)  $ by an arbitrary interval $\left(  0,\gamma\right)  .$
\end{remark}

For $z_{1},z_{2}\in H_{0}^{1}(\Omega)$ we say that $z_{1}\leq z_{2}$ if
$z_{1}\left(  x\right)  \leq z_{2}\left(  x\right)  $ for all $x\in
\lbrack0,1]$.

\begin{theorem}
	Let $u_{0}\leq v_{0}$, $u_{0},v_{0}\in H_{0}^{1}(\Omega)$. If
	$u,v:[0,T]\rightarrow H_{0}^{1}(\Omega)$ are two strong solutions and either
	$u$ or $v$ is degenerate on $[0,T]$, then $u(t)\leq v(t)$ for any $t\in
	\lbrack0,T].$
\end{theorem}

\begin{proof}
	For instance, let $u$ be non-degenerate. Multiplying \eqref{exist2} by
	$(u(t)-v(t))^{+}$ we have%
	\begin{align*}
		&  \frac{1}{2}\frac{d}{dt}\left\Vert (u-v)^{+}\right\Vert _{L^{2}}%
		^{2}+\left\Vert (u-v)^{+}\right\Vert _{H_{0}^{1}}^{2}\\
		&  =\int_{0}^{1}(f_{u}(t,x)-f_{v}(t,x))(u(t,x)-v(t,x))^{+}dx+\omega\left\Vert
		(u-v)^{+}\right\Vert _{L^{2}}^{2},
	\end{align*}
	where $f_{u},\ f_{v}\in L^{\infty}((0,T)\times(0,1))$ and $f_{u}(t,x)\in
	H_{0}(u(t,x)),\ f_{v}(t,x)\in H_{0}(v(t,x))$ for a.a. $\left(  t,x\right)  .$
	Let $L_{\infty}>0$ be such that $\left\Vert z\right\Vert _{L^{\infty}}\leq
	L_{\infty}\left\Vert z\right\Vert _{H_{0}^{1}}$ for $z\in H_{0}^{1}(\Omega)$.
	Hence, Lemma \ref{IneqH0} and $\omega<\pi^{2}$ imply that
	\[
	\frac{1}{2}\frac{d}{dt}\left\Vert (u-v)^{+}\right\Vert _{L^{2}}^{2}%
	\leq(K-\beta^{2})\left\Vert u-v\right\Vert _{L^{\infty}}^{2},
	\]
	where $\beta=(\frac{1}{L_{\infty}})\left(  1-\frac{\omega}{\pi^{2}}\right)
	^{\frac{1}{2}}$, $K=2D>0$ and $D$ is the constant in
	\eqref{Nondegenerate2} for the solution $u$.
	
	If $K\leq\beta^{2}$, then the result follows immediately. Thus, assume that
	$K>\beta^{2}.$
	
	We introduce the rescaling $y=\gamma x$, where $\gamma>0$. We put $u_{\gamma
	}(t,y):=u(t,y/\gamma)$, so \\ $\left(  u_{\gamma}\right)  _{yy}(t,y):=\left(
	u\right)  _{xx}(t,y/\gamma)/\gamma^{2}$. Since%
	\[
	u_{t}(t,\frac{y}{\gamma})-u_{xx}(t,\frac{y}{\gamma})=f_{u}(t,\frac{y}{\gamma
	})+\omega u(t,\frac{y}{\gamma}),\text{ for a.a. }t\in\left(  0,T\right)
	,\text{ }0<y<\gamma,
	\]
	we have%
	\[
	\left(  u_{\gamma}\right)  _{t}(t,y)-\gamma^{2}\left(  u_{\gamma}\right)
	_{yy}(t,y)=f_{u_{\gamma}}(t,y)+\omega u_{\gamma}(t,y),\ \text{for a.a. }%
	t\in\left(  0,T\right)  ,0<y<\gamma,
	\]
	where $f_{u_{\gamma}}(t,y):=f_{u}(t,y/\gamma)$, and the same is true for
	$v_{\gamma}(t,y)=v(t,y/\gamma)$. Thus, $u_{\gamma},v_{\gamma}:[0,T]\rightarrow
	H_{0}^{1}(0,\gamma)$ are strong solutions of the problem%
	\begin{equation}
		\left\{
		\begin{array}
			[c]{l}%
			\dfrac{\partial u}{\partial t}-\gamma^{2}\dfrac{\partial^{2}u}{\partial x^{2}%
			}\in H_{0}(u)+\omega u,\ \text{on\ }(0,\infty)\times(0,\gamma),\\
			u(t,0)=u(t,\gamma)=0,\\
			u(x,0)=u_{0}(x).
		\end{array}
		\right.  \label{EqGamma}%
	\end{equation}

	In what follows, we will use the notation%
	\[
	\left\Vert v\right\Vert _{H^{1}(I_{\gamma})}=\sqrt{\left\Vert v\right\Vert
		_{L^{2}(I_{\gamma})}^{2}+\left\Vert \frac{dv}{dx}\right\Vert _{L^{2}%
			(I_{\gamma})}^{2}}.
	\]

	If $C_{\gamma}$ is the constant in \eqref{Nondegenerate2} for the solution
	$u_{\gamma}$, we need to analyze how it depends on $\gamma$. For the constant
	of nondegeneracy $C$ of $u$ we have%
	\[
	\mu\{x\in\left(  0,1\right)  :\left\vert u(x)\right\vert \leq\alpha\}\leq
	C\alpha,
	\]
	so%
	\[
	\mu\{y\in\left(  0,\gamma\right)  :\left\vert u_{\gamma}(y)\right\vert
	\leq\alpha\}=\int_{\left\vert u_{\gamma}(y)\right\vert \leq\alpha}%
	1\ dy=\int_{\left\vert u(y/\gamma)\right\vert \leq\alpha}1\ dy
	\]%
	\begin{equation}
		=\int_{\left\vert u(x)\right\vert \leq\alpha}\gamma\ dx=\gamma\mu\{x\in\left(
		0,1\right)  :\left\vert u(x)\right\vert \leq\alpha\}\leq\gamma\alpha
		C=C_{\gamma}\alpha, \label{DegeneracyGamma}%
	\end{equation}
	where $C_{\gamma}=\gamma C.$
	
	Let $I_{\gamma}=[0,\gamma]$. We will prove the existence of $\overline
	{L}_{\infty}$ (independent of $\gamma\geq1$) such that%
	\begin{equation}
		\left\Vert w\right\Vert _{L^{\infty}(I_{\gamma})}\leq\overline{L}_{\infty
		}\left\Vert w\right\Vert _{H^{1}(I_{\gamma})}\text{, for any }w\in
		H^{1}(I_{\gamma}). \label{InclusionGamma}%
	\end{equation}
	By [Brezis, Theorem VIII.7] there is a positive constant $\overline{C}$ such
	that
	\[
	\left\Vert v\right\Vert _{H^{1}(\mathbb{R})}\geq\overline{C}\left\Vert
	v\right\Vert _{L^{\infty}(\mathbb{R})}\text{ for all }v\in H^{1}(\mathbb{R}).
	\]
	By [Brezis, Theorem VIII.5] there exists a prolongation operator $P_{\gamma
	}:H^{1}(I_{\gamma})\rightarrow H^{1}(\mathbb{R})$ which satisfies%
	\[
	\left\Vert P_{\gamma}w\right\Vert _{H^{1}(\mathbb{R})}\leq4\left(  1+\frac
	{1}{\gamma}\right)  \left\Vert w\right\Vert _{H^{1}(I_{\gamma})}\text{ for all
	}w\in H^{1}(I_{\gamma}).
	\]
	Also, by the construction it follows that $\left\Vert P_{\gamma}w\right\Vert
	_{L^{\infty}(\mathbb{R})}=\left\Vert w\right\Vert _{L^{\infty}(\mathbb{R})}$.
	Hence, for $\gamma\geq1$, we have%
	\begin{align*}
		\left\Vert u\right\Vert _{H^{1}(I_{\gamma})}  &  \geq\frac{\gamma}%
		{4(1+\gamma)}\left\Vert P_{\gamma}w\right\Vert _{H^{1}(\mathbb{R})}\geq
		\frac{1}{8}\left\Vert P_{\gamma}w\right\Vert _{H^{1}(\mathbb{R})}\\
		&  \geq\frac{\overline{C}}{8}\left\Vert P_{\gamma}w\right\Vert _{L^{\infty
			}(\mathbb{R})}=\frac{\overline{C}}{8}\left\Vert w\right\Vert _{L^{\infty
			}(I_{\gamma})},
	\end{align*}
	so \eqref{InclusionGamma} is true with $\overline{L}_{\infty}=\frac
	{8}{\overline{C}}.$
	
	Multiplying by $(u_{\gamma}(t)-v_{\gamma}(t))^{+}$ the equality%
	\[
	\frac{d}{dt}\left(  u_{\gamma}-v_{\gamma}\right)  -\gamma^{2}\frac
	{\partial^{2}\left(  u_{\gamma}-v_{\gamma}\right)  }{\partial x^{2}%
	}=f_{u_{\gamma}}(t)-f_{v_{\gamma}}(t)+\omega\left(  u_{\gamma}-v_{\gamma
	}\right)  ,
	\]
	where $f_{u_{\gamma}},\ f_{v_{\gamma}}\in L^{\infty}((0,T)\times(0,1))$ are
	such that $f_{u_{\gamma}}(t,x)\in H_{0}(u_{\gamma}(t,x))$,\\ $  f_{v_{\gamma}%
	}(t,x)\in H_{0}(v_{\gamma}(t,x))$ for a.a. $\left(  t,x\right)  $, we obtain%
	\begin{align*}
		&  \frac{1}{2}\frac{d}{dt}\left\Vert (u_{\gamma}-v_{\gamma})^{+}\right\Vert
		_{L^{2}(I_{\gamma})}^{2}+\gamma^{2}\left\Vert (u_{\gamma}-v_{\gamma}%
		)^{+}\right\Vert _{H_{0}^{1}(I_{\gamma})}^{2}\\
		&  \leq\int_{0}^{\gamma}(f_{u_{\gamma}}(t,x)-f_{v_{\gamma}}(t,x))(u_{\gamma
		}(t,x)-v_{\gamma}(t,x))^{+}dx+\frac{\omega\gamma^{2}}{\pi^{2}}\left\Vert
		(u_{\gamma}-v_{\gamma})^{+}\right\Vert _{H_{0}^{1}(I_{\gamma})}^{2}.
	\end{align*}
	Hence, by Remark \ref{IneqH0Remark} and \eqref{DegeneracyGamma} we have%
	\begin{align*}
		&  \frac{1}{2}\frac{d}{dt}\left\Vert (u_{\gamma}-v_{\gamma})^{+}\right\Vert
		_{L^{2}(I_{\gamma})}^{2}+\gamma^{2}\left(  1-\frac{\omega}{\pi^{2}}\right)
		\left\Vert (u_{\gamma}-v_{\gamma})^{+}\right\Vert _{H^{1}(I_{\gamma})}^{2}\\
		& \leq2D_{\gamma}\left\Vert (u_{\gamma}-v_{\gamma}%
		)^{+}\right\Vert _{L^{\infty}(I_{\gamma})}^{2}+\gamma^{2}\left(
		1-\frac{\omega}{\pi^{2}}\right)  \left\Vert (u_{\gamma}-v_{\gamma}%
		)^{+}\right\Vert _{L^{2}(I_{\gamma})}^{2},
	\end{align*}
	where $D_{\gamma}=\max\{C\gamma,1/\alpha_{0}\}$, so%
	\begin{align*}
		&  \frac{1}{2}\frac{d}{dt}\left\Vert (u_{\gamma}-v_{\gamma})^{+}\right\Vert
		_{L^{2}(I_{\gamma})}^{2}\\
		&  \leq\gamma^{2}\left(  1-\frac{\omega}{\pi^{2}}\right)  \left\Vert
		(u_{\gamma}-v_{\gamma})^{+}\right\Vert _{L^{2}(I_{\gamma})}^{2}+\left(
		2C\gamma-\frac{\gamma^{2}\left(  \pi^{2}-\omega\right)  }{\overline{L}%
			_{\infty}^{2}\pi^{2}}\right)  \left\Vert (u_{\gamma}-v_{\gamma})^{+}%
		\right\Vert _{L^{\infty}(I_{\gamma})}^{2}\\
		&  \leq\gamma^{2}\left(  1-\frac{\omega}{\pi^{2}}\right)  \left\Vert
		(u_{\gamma}-v_{\gamma})^{+}\right\Vert _{L^{2}(I_{\gamma})}^{2},
	\end{align*}
	for $\gamma$ great enough. Thus,%
	\[
	\left\Vert (u_{\gamma}-v_{\gamma})^{+}(t)\right\Vert _{L^{2}(I_{\gamma})}%
	^{2}\leq e^{\delta t}\left\Vert (u_{\gamma}-v_{\gamma})^{+}(0)\right\Vert
	_{L^{2}(I_{\gamma})}^{2}=0,
	\]
	for $\delta=2\gamma^{2}\left(  1-\frac{\omega}{\pi^{2}}\right)  $. Hence,
	$u_{\gamma}(t)\leq v_{\gamma}(t)$ and then $u(t)\leq v(t)$ for all
	$t\in\lbrack0,T].$
\end{proof}

\bigskip

\begin{corollary}
	\label{Uniqueness}If $u,v:[0,T]\rightarrow H_{0}^{1}(\Omega)$ are two strong
	solutions such that $u\left(  0\right)  =v\left(  0\right)  =u_{0}$ and either
	$u$ or $v$ is degenerate on $[0,T]$, then $u\left(  t\right)  =v\left(
	t\right)  $ for all $t\in\lbrack0,T]$.
\end{corollary}

\begin{lemma}
	The fixed points $v_{k}^{\pm}$ are nondegenerate.
\end{lemma}

\begin{proof}
	Let us consider first the point $v_{1}^{+}$ and denote $\left(  v_{1}%
	^{+}\right)  ^{\prime}\left(  0\right)  =\gamma_{0}>0$. We choose
	$0<x_{0}<\frac{1}{2}$ such that $\left(  v_{1}^{+}\right)  ^{\prime}\left(
	x\right)  \geq\frac{\gamma_{0}}{2}$ for any $x\in\lbrack0,x_{0}]$. Then for
	$0<\alpha_{0}\leq v_{1}^{+}(x_{0})$ we have%
	\[
	v_{1}^{+}\left(  x\right)  \geq\alpha_{0}\text{ }\forall x\in\lbrack
	x_{0},1-x_{0}],
	\]%
	\[
	v_{1}^{+}(1-x)=v_{1}^{+}\left(  x\right)  =\left(  v_{1}^{+}\right)  ^{\prime
	}(\overline{x})x\geq\frac{\gamma_{0}}{2}x\ \forall x\in\lbrack0,x_{0}].
	\]
	Hence, for $0<\alpha<\alpha_{0}$ and $x\in\lbrack0,x_{0}]$ such that
	$v_{1}^{+}\left(  x\right)  \leq\alpha$ we obtain that%
	\[
	\frac{\gamma_{0}}{2}x\leq\alpha,
	\]
	so by symmetry,%
	\[
	\mu\{x\in\left(  0,1\right)  :v_{1}^{+}\left(  x\right)  \leq\alpha\}\leq
	\frac{4}{\gamma_{0}}\alpha.
	\]

	Therefore, $v_{1}^{+}$ is nondegenerate and then so is $v_{1}^{-}.$
	
	By symmetry, we easily deduce that
	\[
	\mu\{x\in\left(  0,1\right)  :\left\vert v_{k}^{+}\left(  x\right)
	\right\vert \leq\alpha\}\leq\frac{4k}{\gamma_{0}}\alpha,
	\]
	where $\left(  v_{k}^{+}\right)  ^{\prime}\left(  0\right)  =\gamma_{0}$ and
	$x_{0},\ \alpha_{0}$ are such that $0<x_{0}<\frac{1}{2k}$, $\left(  v_{k}%
	^{+}\right)  ^{\prime}\left(  x\right)  \geq\frac{\gamma_{0}}{2},$ for any
	$x\in\lbrack0,x_{0}],$ and $0<\alpha_{0}\leq v_{k}^{+}(x_{0}).$
	
	Thus, $v_{k}^{\pm}$ are nondegenerate.
\end{proof}

\bigskip

From the previous corollary and the fact that the fixed points $v_{k}^{\pm}$
are nondegenerate we obtain the following result.

\begin{corollary}
	For any fixed point $v_{k}^{\pm}$, $k\geq1$, the solution $u\left(
	\text{\textperiodcentered}\right)  $ with initial condition $u\left(
	0\right)  =v_{k}^{\pm}$ is unique on $[0,+\infty)$.
\end{corollary}

Finally, We will define a suitable neighborhood of the point $v_{k}^{\pm}$
where all the solutions are uniquely defined. We consider the space
$X=V^{2r}=D\left(  A^{r}\right)  $ with $\frac{3}{4}<r<1$. We know that $X$ is
continuously embedded into the space $C^{1}([0,1]).$ We denote by $O_{\delta
}(v_{0})=\{v\in X:\left\Vert v-v_{0}\right\Vert _{X}<\delta\}$ a $\delta
$-neighborhood of the point $v_{0}\in X.$

\begin{lemma}
	\label{NeighborhoodDegenerate}For any $v_{k}^{\pm}$ there exist $\delta
	,C,\alpha_{0}>0$ such that
	\[
	\mu\{x\in\left(  0,1\right)  :\left\vert v\left(  x\right)  \right\vert
	\leq\alpha\}\leq C\alpha\text{ }\forall v\in O_{\delta}(v_{k}^{\pm}%
	),\ \alpha\in(0,\alpha_{0}).
	\]
	
\end{lemma}

\begin{remark}
	This result means that the functions are uniformly nondegenerate in some
	neighborhood $O_{\delta}(v_{k}^{\pm}).$
\end{remark}

\begin{proof}
	We will analyze the function $v_{2}^{-}$. The proof is rather similar for the
	other points.
	
	Denote $\gamma_{0}=\left(  v_{2}^{-}\right)  ^{\prime}\left(  \frac{1}%
		{2}\right)  >0$. Since $X\subset C^{1}([0,1])$, we can choose $\delta>0,$
	$0<x_{0}<\frac{1}{2}$ such that any $v\in O_{\delta_{1}}(v_{2}^{-})$ satisfies:
	
	\begin{itemize}
		\item $v$ has only one zero $x_{v}$ in $\left(  0,1\right)  $ and $x_{v}%
		\in\left(  \frac{1}{2}-x_{0},\frac{1}{2}+x_{0}\right)  ;$
		
		\item $v^{\prime}\left(  x\right)  \geq\frac{\gamma_{0}}{2}$ for all
		$x\in\lbrack\frac{1}{2}-x_{0},\frac{1}{2}+x_{0}];$
		
		\item $v^{\prime}(x)\leq-\frac{\gamma_{0}}{2}$ for all $x\in\lbrack
		0,x_{0}]\cup\lbrack1-x_{0},1];$
		
		\item $\left\vert v\left(  x\right)  \right\vert \geq\alpha_{0}$ for all
		$x\in\lbrack x_{0},\frac{1}{2}-x_{0}]\cup\lbrack\frac{1}{2}+x_{0},1-x_{0}], $
	\end{itemize}
	
	where $\alpha_{0}<v_{2}^{-}(\frac{1}{2}+x_{0})$. Hence,
	\begin{align*}
		v(x)  &  =v^{\prime}(x_{1}^{v})(x-x_{v})\geq\frac{\gamma_{0}}{2}%
		(x-x_{v}),\text{ for }x\in\lbrack x_{v},\frac{1}{2}+x_{0}],\\
		v(x)  &  =v^{\prime}(x_{2}^{v})(x-x_{v})\leq\frac{\gamma_{0}}{2}%
		(x-x_{v}),\text{ for }x\in\lbrack\frac{1}{2}-x_{0},x_{v}],\\
		v(x)  &  =v^{\prime}(x_{3}^{v})x\leq-\frac{\gamma_{0}}{2}x\text{, for }%
		x\in\lbrack0,x_{0}],\\
		v(x)  &  =v^{\prime}(x_{4}^{v})(x-1)\geq\frac{\gamma_{0}}{2}(1-x)\text{, for
		}x\in\lbrack1-x_{0},1].
	\end{align*}
	Therefore,%
	\[
	\mu\{x\in\left(  0,1\right)  :\left\vert v\left(  x\right)  \right\vert
	\leq\alpha\}\leq\frac{8}{\gamma_{0}}\alpha.
	\]
	
\end{proof}

\bigskip

Let $u_{0}\in O_{\delta}(v_{k}^{\pm})$ and $u\left(  \text{\textperiodcentered
}\right)  \in\mathcal{D}(u_{0})$. Let $T_{\max}$ be the maximal time such that
$u\left(  t\right)  \in O_{\delta}(v_{k}^{\pm})$ for all $t\in\lbrack
0,T_{\max})$. Then from Lemmas \ref{Uniqueness}, \ref{NeighborhoodDegenerate}
we deduce that $u\left(  \text{\textperiodcentered}\right)  $ is the unique
solution on $[0,T_{\max})$ (and if $T_{\max}<\infty$, it is the unique
solution on $[0,T_{\max}]$).

\begin{remark}
	For the semigroup defined on $O_{\delta}(v_{k}^{\pm})$ we could apply Theorem
	5.1 from \cite{Rybakowski} in order to obtain the existence of an isolating block.
\end{remark}

\section{Conclusions}

In this paper, we have presented an abstract result proving the existence of
isolating blocks for multivalued semiflows. Hence, given an isolated weakly
invariant set defined for a multivalued semiflow satisfying (K1)-(K5), we can
find a special neighborhood for which the boundaries are completely oriented
in some sense. We believe that our construction of isolating blocks for
multivalued semiflows is the first of its kind, so as the application to
differential inclusions.

In the single-valued case, such neighborhood of an isolated weakly invariant
set is essential and gives the inspiration for the definition of Conley's
index. It can be shown that the isolating block together with its boundary has
the cofibration property. In fact, the quotient space defined by the isolating
block over its boundary is the Conley index. This is also true in the context
of metric spaces which are not necessarily locally compact, see \cite[Theorem
5.1]{Rybakowski}.

Having the concept of isolating block a very close relation with Conley's
index, we may wonder if we can define a homology index for multivalued
semiflows. This is a subject for further studies. It is important to say that
there are already very nice and interesting works that propose some
definitions of Conley's index in the multivalued setting, see e.g.
\cite{Mrozek90_CI} and \cite{DzGg_CI_mult_Hilb}. Once we will have a candidate
for the definition of Conley's index, we will try to understand if we are able
to present something new with that definition and which is the relation of
this new concept with the ones proposed in the previous works. We hope to
achieve an answer to these questions in the future.

%

\section*{Acknowledgements}
This work was carried out during a visit of the first author to the Centro de
Investigaci\'{o}n Operativa (CIO), UMH de Elche. She wishes to express her
gratitude to the people from CIO for the warm reception and kindness.

%
%

\bibliographystyle{plain}
\bibliography{biblio}

\begin{thebibliography}{10}

\bibitem{ARV06}
Jos\'{e}~M. Arrieta, An\'{\i}bal Rodr\'{\i}guez-Bernal, and Jos\'{e} Valero.
\newblock Dynamics of a reaction-diffusion equation with a discontinuous
  nonlinearity.
\newblock {\em Internat. J. Bifur. Chaos Appl. Sci. Engrg.}, 16(10):2965--2984,
  2006.

\bibitem{Ball}
John Ball.
\newblock Continuity properties and global attractors of generalized semiflows
  and the navier-stokes equations.
\newblock {\em J. Nonlinear Sci.}, 7:475--502, 1997.

\bibitem{Barbu}
Viorel Barbu.
\newblock {\em Nonlinear semigroups and differential equations in {B}anach
  spaces}.
\newblock Editura Academiei Republicii Socialiste Rom\^{a}nia, Bucharest;
  Noordhoff International Publishing, Leiden, 1976.
\newblock Translated from the Romanian.

\bibitem{BzGaRs_CI_Hilb}
Zbigniew B\l{a}szczyk, Anna Go\l\k{e}biewska, and S\l{a}womir Rybicki.
\newblock Conley index in {H}ilbert spaces versus the generalized topological
  degree.
\newblock {\em Adv. Differential Equations}, 22(11-12):963--982, 2017.

\bibitem{Conley}
Charles Conley.
\newblock {\em Isolated invariant sets and the {M}orse index}, volume~38 of
  {\em CBMS Regional Conference Series in Mathematics}.
\newblock American Mathematical Society, Providence, R.I., 1978.

\bibitem{HenVal}
Henrique~B. da~Costa and Jos\'{e} Valero.
\newblock Morse decompositions and lyapunov functions for dynamically gradient
  multivalued semiflows.
\newblock {\em Nonlinear Dyn.}, 84:19--34, 2016.

\bibitem{DKP}
Sergey Dashkovskyi, Oleksiy Kapustyan, and Yuriy Perestyuk.
\newblock Stability of uniform attractors of impulsive multi-valued semiflows.
\newblock {\em Nonlinear Anal. Hybrid Syst.}, 40, 2021.

\bibitem{DzGg_CI_mult_Hilb}
Zdzis\l{a}w Dzedzej and Grzegorz Gabor.
\newblock On homotopy {C}onley index for multivalued flows in {H}ilbert spaces.
\newblock {\em Topol. Methods Nonlinear Anal.}, 38(1):187--205, 2011.

\bibitem{FeNo}
Eduard Feireisl and John Norbury.
\newblock Some existence, uniqueness and nonuniqueness theorems for solutions
  of parabolic equations with discontinuous nonlinearities.
\newblock {\em Proc. Roy. Soc. Edinburgh Sect. A}, 119(1-2):1--17, 1991.

\bibitem{GIP_CI_Hilb}
K.~G\k{e}ba, M.~Izydorek, and A.~Pruszko.
\newblock The {C}onley index in {H}ilbert spaces and its applications.
\newblock {\em Studia Math.}, 134(3):217--233, 1999.

\bibitem{IzMaSta_CI_Hilb}
Marek Izydorek, Thomas~O. Rot, Maciej Starostka, Marcin Styborski, and Robert
  C. A.~M. Vandervorst.
\newblock Homotopy invariance of the {C}onley index and local {M}orse homology
  in {H}ilbert spaces.
\newblock {\em J. Differential Equations}, 263(11):7162--7186, 2017.

\bibitem{Janig_NonAutCI}
Axel J\"{a}nig.
\newblock Nonautonomous {C}onley index theory. {T}he homology index and
  attractor-repeller decompositions.
\newblock {\em Topol. Methods Nonlinear Anal.}, 53(1):57--77, 2019.

\bibitem{KKV14}
Oleksiy~V. Kapustyan, Pavlo~O. Kasyanov, and Jos\'{e} Valero.
\newblock Structure and regularity of the global attractor of a
  reaction-diffusion equation with non-smooth nonlinear term.
\newblock {\em Discrete Contin. Dyn. Syst.}, 34(10):4155--4182, 2014.

\bibitem{MelnikValero98}
Valery~S. Melnik and José Valero.
\newblock {O}n attractors of multivalued semiflows and differential inclusions.
\newblock {\em Set-Valued Anal.}, 6:83--111, 1998.

\bibitem{Mischaikow95}
Konstantin Mischaikow.
\newblock Global asymptotic dynamics of gradient-like bistable equations.
\newblock {\em SIAM J. Math. Anal.}, 26(5):1199--1224, 1995.

\bibitem{Mrozek90_CI}
Marian Mrozek.
\newblock A cohomological index of {C}onley type for multi-valued admissible
  flows.
\newblock {\em J. Differential Equations}, 84(1):15--51, 1990.

\bibitem{NC}
Gerald~R. North and Robert~F. Cahalan.
\newblock Predictability in a solvable stochastic climate model.
\newblock {\em Siam J. Math. Anal.}, J. Atmos. Sci.:504--513, 1982.

\bibitem{Rybakowski}
Krzysztof~P. Rybakowski.
\newblock {\em The homotopy index and partial differential equations}.
\newblock Universitext. Springer-Verlag, Berlin, 1987.

\bibitem{Terman1}
David Terman.
\newblock A free boundary problem arising from a bistable reaction-diffusion
  equation.
\newblock {\em Siam J. Math. Anal.}, 14(6):1107--1129, 1983.

\bibitem{Terman2}
David Terman.
\newblock A free boundary arising from a model for nerve conduction.
\newblock {\em J. Differential Equations}, 58(3):345--363, 1985.

\bibitem{Tolstonogov}
A.~A. Tolstonogov.
\newblock Solutions of evolution inclusions. {I}.
\newblock {\em Sibirsk. Mat. Zh.}, 33(3):161--174, 221, 1992.

\bibitem{Valero01}
Jos\'{e} Valero.
\newblock Attractors of parabolic equations without uniqueness.
\newblock {\em J. Dynam. Differential Equations}, 13(4):711--744, 2001.

\bibitem{Valero05}
Jos\'{e} Valero.
\newblock On the {K}neser property for some parabolic problems.
\newblock {\em Topology Appl.}, 153(5-6):975--989, 2005.

\bibitem{Valero2021}
Jos\'{e} Valero.
\newblock Characterization of the attractor for nonautonomous
  reaction-diffusion equations with discontinuous nonlinearity.
\newblock {\em J. Differential Equations}, 275:270--308, 2021.

\bibitem{ZKKVZ}
Mikail~Z. Zgurovsky, Kasyanov~Pavlo O., Oleksiy~V. Kapustyan, Jos\'e Valero,
  and Nina~V. Zadoianchuk.
\newblock {\em Evolution Inclusions and variation inequalities for earth data
  processing III}, volume 211.
\newblock Springer-Verlag, Berlin, 2012.

\end{thebibliography}

\end{document}